\documentclass[12pt]{amsart}
\newif\ifdraft
\draftfalse

\usepackage{color}
\ifdraft\usepackage[notref,notcite]{showkeys}\fi
\usepackage{graphicx,subfigure,suffix}
\usepackage{indentfirst}
\usepackage{bm, mathrsfs, graphics,subeqnarray,setspace,amsthm,epstopdf}
\usepackage{amssymb,amsmath,amsfonts,amssymb,mathtools}
\usepackage{latexsym}
\usepackage{cite}

-\marginparwidth \oddsidemargin -9mm \evensidemargin \oddsidemargin
\advance\hoffset by 5mm

\allowdisplaybreaks \numberwithin{equation}{section}

\newtheorem{thm}{Theorem}[section]
\newtheorem{lmm}[thm]{Lemma}
\newtheorem{pro}[thm]{Proposition}
\newtheorem{cor}[thm]{Corollary}

\catcode`,\active

\catcode`\,12

\theoremstyle{definition}
\newtheorem{definition}[thm]{Definition}

\theoremstyle{remark}
\newtheorem{rmk}[thm]{Remark}
\newtheorem*{remark*}{Remark}

\newcommand{\Chi}[1]{\Chi*{\{#1\}}}
\WithSuffix\newcommand\Chi*[1]{\chi_{\raise-.5ex\hbox{$\scriptstyle#1$}}}

\renewcommand{\leq}{\leqslant}
\renewcommand{\geq}{\geqslant}
\renewcommand{\ge}{\geqslant}
\renewcommand{\le}{\leqslant}

\newcommand\blfootnote[1]{%
  \begingroup
  \renewcommand\thefootnote{}\footnote{#1}%
  \addtocounter{footnote}{-1}%
  \endgroup
}

\makeatletter
\@mparswitchfalse
\makeatother

\begin{document}
\title{Continuous and discrete one dimensional autonomous fractional ODEs}

\begin{abstract}
 In this paper, we study 1D autonomous fractional ODEs $D_c^{\gamma}u=f(u), 0< \gamma <1$, where $u: [0,\infty)\mapsto\mathbb{R}$ is the unknown function and $D_c^{\gamma}$ is the generalized Caputo derivative introduced by Li and Liu (	arXiv:1612.05103). Based on the existence and uniqueness theorem and regularity results in previous work, we show the monotonicity of solutions to the autonomous fractional ODEs and several versions of comparison principles.  We also perform a detailed discussion of the asymptotic behavior for $f(u)=Au^p$. In particular, based on an Osgood type blow-up criteria, we find relatively sharp bounds of the blow-up time in the case $A>0, p>1$. These bounds indicate that as the memory effect becomes stronger ($\gamma\to 0$), if the initial value is big, the blow-up time tends to zero while if the initial value is small, the blow-up time tends to infinity. In the case $A<0, p>1$, we show that the solution decays to zero more slowly compared with the usual derivative.  Lastly, we show several comparison principles and Gr\"{o}nwall inequalities for discretized equations, and perform some numerical simulations to confirm our analysis.
\end{abstract}
\author{Yuanyuan Feng}
\address{\hskip-\parindent
Yuanyuan Feng\\
Department of Mathematical Sciences\\
Carnegie Mellon University\\
Pittsburgh, PA 15213, USA
}
\email{yuanyuaf@andrew.cmu.edu}

\author{Lei Li}
\address{\hskip-\parindent
Lei Li\\
Department of Mathematics\\
Duke University\\
Durham, NC 27708, USA
}
\email{leili@math.duke.edu}

\author{Jian-Guo Liu}
\address{\hskip-\parindent
Jian-Guo Liu\\
Departments of Physics and Mathematics\\
Duke University\\
Durham, NC 27708, USA
}
\email{jliu@phy.duke.edu}

\author{Xiaoqian Xu}
\address{\hskip-\parindent
Xiaoqian Xu\\
Department of Mathematical Sciences\\
Carnegie Mellon University\\
Pittsburgh, PA 15213, USA
}
\email{xxu@math.cmu.edu}

\maketitle

\section{Introduction}
\blfootnote{2010 {\it Mathematics Subject Classification}. Primary 34A08.} \blfootnote{{\it Key words and phrases.} Fractional ODE, Caputo derivative, Volterra integral equation, blow-up time, discrete Gr\"{o}nwall inequality.}
The fractional calculus in continuous time has been used widely in physics and engineering for memory effect, viscoelasticity, porous media etc \cite{gm97, df02, kst06,mpg07,diethelm10, ala16, taylorremarks}. There are two types of fractional derivatives that are commonly used: the Riemann-Liouville derivatives and the Caputo derivatives (See \cite{kst06}). 

The Riemann-Liouville derivatives are named after Bernhard Riemann and Joseph Liouville. Liouville was the first to study fractional derivative rigorously (see, for example, \cite{munkhammar2004riemann,drapaca2012fractional} for a better survey). On the other hand, the Caputo's definition of fractional derivatives was first introduced in \cite{caputo1967linear} to study the memory effect of energy dissipation for some anelastic materials, and soon became a useful modeling tool  in engineering and physical sciences to construct physical models for nonlocal interactions in time (see \cite{zaslavsky2002chaos}). 

Compared with Riemann-Liouville derivatives, Caputo derivatives remove the singularities at the origin and have many properties that are similar to the ordinary derivative so that they are suitable for initial value problems \cite{liliu16}. However, the classical $\gamma$-th Caputo derivative of a function requires an integer order derivative no less than $\gamma$, which seems to be artificial. In \cite{ala16}, Allen, Caffarelli and Vasseur have introduced an alternative form of Caputo derivatives to avoid using higher regularity of the function. In \cite{liliu16}, another extension of Caputo derivatives was proposed, by which the higher derivative of the function is not needed either and can recover the definition in \cite{ala16}. Moreover, this new definition allows us to transform fractional ODEs into Volterra type integral equations,  by deconvolution through an underlying group property without the higher regularity assumption. This provides a convenient framework for us to study the fractional ODEs with Caputo derivatives. 

There is a huge amount of literature discussing fractional differential equations. However, few of them discuss the behavior of the solutions to fractional ODEs systematically. For reference, some results can be found in \cite{df02, diethelm10} using the traditional Caputo derivatives. 

In this paper, we use the new definition of Caputo derivative in \cite{liliu16} (also see Definition \ref{def:caputo} below) to make a detailed investigation of the nonlinear fractional ODE 
\begin{equation}\label{eq:ode1}
D_c^{\gamma}u= f(u), ~ u(0)=u_0,
\end{equation}
for $\gamma\in (0, 1)$. Here $f$ is locally Lipschitz whose domain contains $u_0$, and $D_c^{\gamma}$ represents the Caputo derivative of order $\gamma$. In the rest of this paper we will assume $u_0\ge 0$ without loss of generality (if $u_0<0$, we can do change of variables $v=-u$ and study $D_c^{\gamma}v=\tilde{f}(v)$ where $\tilde{f}(v):=-f(-v)=-f(u)$). Studying the behavior of the solution to this fractional ODE is important for the analysis of fractional partial differential equations (fractional PDEs), as we usually need {\it a priori} estimates of certain energies of the solution to a fractional PDE, which have form
\[
D_c^{\gamma}E\le AE^p.
\]
By the comparison principles in \cite{liliu16} or in Section \ref{sec:comp}, the energy norm may be controlled by the solution of the fractional ODE \eqref{eq:ode1}. Hence, we will focus on the particular cases $f(u)=Au^p$ in detail.

According to \cite{liliu16}, the fractional ODE \eqref{eq:ode1} is equivalent to a Volterra type integral equation without assuming high regularity of the solution, which is the important starting point for our study. It is well-known that the solutions of 1D autonomous ODEs with usual first order derivative are monotone,  since the solution curves never cross zeros of $f$ and $f(u)$ has a definite sign. One of our main results is that if $f\in C^1$ and $f'$ is locally Lipschitz, the first order derivative of the solution to the fractional ODE \eqref{eq:ode1} does not change sign and therefore the solution is monotone (see Theorem \ref{thm:mon}). This is based on Lemma \ref{lmm:pos}, which is a slightly different version of \cite[Theorem 1]{weis75}. Lemma \ref{lmm:pos} ensures the positivity of the solutions of the integral equation that $y=u'$ or $y=-u'$ satisfies:
\begin{gather*}
y(t)+\int_0^t(t-s)^{\gamma-1}v(s)y(s)ds=\alpha t^{\gamma-1}, ~\alpha>0,
\end{gather*}
where $v$ is continuous. The idea is to use the resolvent for the kernel $\lambda t^{\gamma-1}$ to transform this integral equation into another integral equation (see \eqref{eq:alterinteq}) so that all the functions involved are non-negative. The solution to the new integral equation \eqref{eq:alterinteq} is nonnegative, implying that the first derivative of the solution to \eqref{eq:ode1} does not change sign. 

Another contribution of this paper is to discuss the special cases $f(u)=Au^p$ in detail and to reveal several interesting roles of memory. In particular, for the cases $A>0, p>1$, we find relatively sharp estimates of the blow-up time $T_b$. The lower bound of $T_b$ is important for the inequality $D_c^{\gamma}E\le AE^p$ since it ensures that $E$ is defined and controlled by the solution of \eqref{eq:ode1} up to this lower bound. Through these bounds,  we find that there exist $u_{02}>u_{01}>0$ so that if $u_0<u_{01}$, the blow-up time $T_b\to +\infty$ as $\gamma\to 0$ (the memory becomes stronger) and if $u_0>u_{02}$, $T_b\to 0$, as $\gamma\to 0$ (See Theorem \ref{buT}). For the cases $A<0, p>1$, we show that under the memory, the solution decays to zero much more slowly compared with the usual ODE (see Theorem \ref{thm:slowdecay}).

By discretizing the differential equation \eqref{eq:ode1} or the equivalent integral equation, we obtain two classes of numerical schemes or discrete equations. Using some discrete comparison principles, we show that if $f$ is nonnegative, nondecreasing, then the numerical solutions to the explicit schemes for the integral equation are absolutely stable: $u^n \le u(nk)$ where $k$ is the time step (Proposition \ref{pro:intnum}). 
In the case $f$ is nonnegative, nondecreasing and the solution to \eqref{eq:ode1} is convex, we prove that the numerical solutions to the explicit schemes for the differential equation are also absolutely stable: $u_{ex}^n\le u(nk)$ and  the numerical solutions to the implicit schemes for the differential equation are bounded below as $u_{im}^n\ge u(nk)$ (Theorem \ref{thm:diffnum}). Hence, the explicit schemes may be used to prove the stability and convergence of some approximation schemes for fractional PDEs and thus the convergence and existence of solutions. The implicit schemes may be used to prove positivity of solutions and to estimate the blow-up time. 
 
 The rest of the paper is organized as follows: In Section \ref{sec:pre}, we introduce the basic definitions, notations and results that are mainly established in \cite{liliu16}. In Section \ref{sec:basicpro}, we study the basic properties of the solutions. In particular, (1) given $f(u)$ is smooth, the solutions are smooth in $(0,\infty)$ but only $\gamma$-H\"older continuous at $t=0$; (2) the solutions are monotone on the interval of existence; (3) an Osgood type finite time blow-up criteria holds provided that $f(u)$ is positive nondecreasing.  In Section \ref{sec:comp}, we prove several comparison principles. In Section \ref{sec:specialcase}, we study the special cases $f(u)=Au^p$. More precisely, for $A>0, p>1$, we provide relatively sharp bounds for the blow-up time, while for $A<0, p>1$, we show the slow decaying as $t\to\infty$. These discussions reveal the roles of memory introduced by fractional derivatives. Lastly, in Section \ref{sec:discrete}, we discuss the discrete equations. To be more specific, we show several discrete comparison principles and use them to study some explicit and implicit schemes. Some numerical simulations are then performed using these schemes to verify our analysis for the continuous cases.

\section{Preliminaries}\label{sec:pre}
In this section we collect some notations and definitions we will use in this paper. 
\subsection{Fractional derivatives}
First, let us make a brief introduction of the definition of fractional derivatives. Before we state the definition, we need the following clarification of notation:
\begin{definition}\label{def:rightlimit}
For a locally integrable function $u\in L_{loc}^1(0, T)$, if there exists $u_0\in\mathbb{R}$ such that 
\begin{gather}
\lim_{t\to 0+}\frac{1}{t}\int_0^t|u(s)-u_0|ds=0,
\end{gather}
we call $u_0$ the right limit of $u$ at $t=0$, denoted as $u(0+):=u_0$.
\end{definition}

As in \cite{liliu16}, we use the following distributions $\{g_\beta\}$ as the convolution kernels for $\beta>-1$:
\begin{gather*}
g_{\beta}=\displaystyle
\begin{cases}
\frac{\theta(t)}{\Gamma(\beta)}t^{\beta-1},& \beta>0,\\
\delta(t), &\beta=0,\\
\frac{1}{\Gamma(1+\beta)}D\left(\theta(t)t^{\beta}\right), & \beta\in (-1, 0).
\end{cases}
\end{gather*} 
Here $\theta(t)$ is the standard Heaviside step function, $\Gamma(\gamma)$ is the gamma function, and $D$ means the distributional derivative. 

$g_{\beta}$ can also be defined for $\beta\le -1$ (see \cite{liliu16}) so that these distributions form a convolution group $\{g_{\beta}: \beta\in\mathbb{R}\}$, and consequently we have 
\begin{gather}\label{eq:group}
g_{\beta_1}*g_{\beta_2}=g_{\beta_1+\beta_2},
\end{gather}
where the convolution between distributions with special non-compact supports is defined through the partition of unit of $\mathbb{R}$.

Now we are able to give the definition of the fractional derivatives.
\begin{definition}\label{def:caputo}
Let $0<\gamma<1$. Consider $u\in L_{loc}^1(0, T)$ that has a right limit $u(0+)$ at $t=0$ in the sense of Definition \ref{def:rightlimit}. The $\gamma$-th order Caputo derivative of $u$ is a distribution in $\mathscr{D}'(-\infty, T)$ with support in $[0, T)$, given by 
\[
D_c^{\gamma}u=g_{-\gamma}*\Big(\theta(t)u\Big)-u(0+)g_{1-\gamma}
=g_{-\gamma}*\Big((u-u(0+))\theta(t)\Big).
\]
\end{definition}
\begin{rmk}
In the case $T=\infty$, the convolution $g_{-\gamma}*u$ is defined through partition of unit of $\mathbb{R}$. In the case of $T<\infty$, $g_{-\gamma}*u$ should be understood as the restriction of the convolution onto $\mathscr{D}'(-\infty,T)$. One can refer to \cite{liliu16} for the technical details. 
\end{rmk}
\begin{rmk}
As discussed in \cite{liliu16}, if there is a version of $u$ (i.e. modifying $u$ on a Lebesgue measure zero set) that is absolutely continuous on $(0, T)$, which is denoted as $u$ again,  then the Caputo derivative is reduced to 
\begin{gather}\label{eq:captra}
D_c^{\gamma}u=\frac{1}{\Gamma(1-\gamma)}\int_0^t\frac{u'(s)}{(t-s)^{\gamma}}ds,
\end{gather}
which is the traditional definition of Caputo derivative. Whenever $u$ is $\gamma+\delta$-H\"older continuous ($\forall \delta>0$), we have
\begin{gather}\label{eq:capint}
D_c^{\gamma}u=\frac{1}{\Gamma(1-\gamma)}\left(\frac{u(t)-u(0)}{t^{\gamma}}+\gamma \int_0^t\frac{u(t)-u(s)}{(t-s)^{\gamma+1}}ds\right).
\end{gather}
Equation \eqref{eq:capint} is the definition for the Caputo derivative used in \cite{ala16}. Intuitively, \eqref{eq:capint} is obtained by integration by parts from \eqref{eq:captra}.
\end{rmk}
Definition \ref{def:caputo} is more useful than the traditional definition (Equation \eqref{eq:captra}) (see for instance \cite{gm97, df02, kst06,mpg07,diethelm10, taylorremarks}) theoretically, since it asks for little regularity and reveals the underlying group structure. With the assumption that $u$ is locally integrable and has a right limit at $t=0$, Definition \ref{def:caputo} and the group property \eqref{eq:group} allow one to convert \eqref{eq:ode1} into the integral form 
\begin{equation}\label{eq:vol}
u(t)=u(0+)+g_{\gamma}*\Big(\theta(t)f(u(t))\Big)=u(0+)+\frac{1}{\Gamma(\gamma)}\int_0^t(t-s)^{\gamma-1}f(u(s))ds,
~\forall t\in (0, T).
\end{equation}

\begin{rmk}
Obtaining Equation \eqref{eq:vol} from the traditional Caputo derivative \eqref{eq:captra} needs us to assume in advance that the unknown function $u$ has too much regularity (for example, for the definition to make sense, we have to assume that $u'(s)$ exists). Using the new definition of Caputo derivative in \cite{liliu16}, the integral form \eqref{eq:vol} is equivalent to Equation \eqref{eq:ode1}  with the assumption that $u$ is locally integrable and has a right limit at $t=0$ in the sense of Definition \ref{def:rightlimit} only. One can check \cite{df02, kst06, diethelm10,liliu16} for more details.
\end{rmk}

Equation \eqref{eq:vol} is called Volterra integral equation, which has been studied extensively.  Analysis of Equation \eqref{eq:ode1} (or equivalently \eqref{eq:vol}) can help us understand the time-delay properties of Caputo derivatives.

\subsection{Existence and uniqueness of solutions to \eqref{eq:ode1}}
In this paper, we will use the following definition of solutions:
\begin{definition}
$u(\cdot)\in L^1(0, T)$ that has a right limit at $t=0$ in the sense of Definition \ref{def:rightlimit} is called a weak solution of \eqref{eq:ode1} if the equation is satisfied in the distribution sense and $u(0+)=u_0$. A weak solution $u$ is called a strong solution if $D_c^{\gamma}u \in L^1(0, T)$ and \eqref{eq:ode1} is satisfied almost everywhere with respect to Lebesgue measure. 
\end{definition}
By the equivalence of \eqref{eq:ode1} and \eqref{eq:vol} established in \cite{liliu16}, all weak solutions of \eqref{eq:ode1} satisfy the integral equation \eqref{eq:vol} almost everywhere with respect to Lebesgue measure.
By modifying the result in \cite[Theorem 6]{liliu16}, we have the following proposition:
\begin{pro}\label{pro:exist}
If $f(u)$ is locally Lipschitz continuous on an interval $(\alpha, \beta)\subset \mathbb{R}$, then $\forall u_0\in (\alpha, \beta)$, there is a unique continuous strong solution with $u(0)=u_0$. Either this solution exists globally on $[0,\infty)$ or there exists $T_b>0$ such that either \[
\liminf_{t\to T_b^-}u(t)=\alpha 
\] 
or 
\[
\limsup_{t\to T_b^-}u(t)=\beta.
\]
\end{pro}
The claim is essentially the same as \cite[Theorem 6]{liliu16}, so we omit the proof.

\begin{definition}
If $u(\cdot)$ exists globally, we set $T_b=\infty$ (See Proposition \ref{pro:exist}). In the case that $\max(|\alpha|, |\beta|)=\infty$ and $\limsup_{t\to T_b^-}|u(t)|=\infty$, we call $T_b$ the blow-up time.
\end{definition}

\section{Some Basic properties of solutions to \eqref{eq:ode1}}\label{sec:basicpro}
\subsection{Regularity and monotonicity of solutions}

In this subsection, we present and prove the regularity and monotonicity results of solutions to \eqref{eq:ode1}. Lemma \ref{lmm:regu} is the result proved in \cite{mf71} for integral equations. This lemma gives the regularity of the solutions to \eqref{eq:ode1} and lays the foundation for our later discussion. Theorem \ref{thm:mon} is the main result in this subsection, which states that the solutions of the autonomous equations are generally monotone. The proof of this theorem relies on Lemma \ref{lmm:pos}, which ensures the positivity of the solutions to a certain class of integral equations. Lemma \ref{lmm:pos} is a slightly different version of \cite[Theorem 1]{weis75}: the author of \cite{weis75} assumed $y$ and $h$ to be continuous at $t=0$ but we cannot assume this for our purpose. However, the idea of the proof is the same.

We present the regularity lemma (\cite[Theorem 1]{mf71}):
\begin{lmm}\label{lmm:regu}
Suppose $f\in C^{1}(\alpha, \beta)$ for some interval $(\alpha, \beta)$ and $f'$ is locally Lipschitz on $(\alpha, \beta)$. Let $u$ be the unique solution to \eqref{eq:ode1} with $u_0\in (\alpha, \beta)$. Then, $u\in C^0[0, T_b)\cap C^{1}(0, T_b)$. Moreover, $y=u'$ satisfies the integral equation
\begin{gather}\label{eq:der}
y(t)=\frac{f(u_0)}{\Gamma(\gamma)}t^{\gamma-1}+\frac{1}{\Gamma(\gamma)}\int_0^t s^{\gamma-1}f'(u(t-s))y(t-s) ds,~\forall t\in (0, T_b).
\end{gather}
  As $t\to 0^+$, $u'(t)=O(t^{\gamma-1})$. If $f(u_0)\neq 0$, $u'(t)\sim \frac{f(u_0)}{\Gamma(\gamma)}t^{\gamma-1}$ as $t\to 0^+$.
\end{lmm}
For the idea of proof, one may fix $T\in (0, T_b)$ and show that \eqref{eq:der} has a unique continuous solution on $(0, T)$. Then using the equation for $(u(t+h)-u(t))/h$, one can verify that this finite difference converges to the solution of \eqref{eq:der}. One can refer to \cite{mf71} for a detailed discussion. For the last claim, as long as we have $u'=O(t^{\gamma-1})$, we can show that the integral is then dominated by the first term as $t\to 0^+$.

\begin{rmk}
Using the group property $g_{n\gamma}*g_{\gamma}=g_{(n+1)\gamma}$ (Equation \eqref{eq:group}), we may find that the solution to \eqref{eq:ode1} is a power series of $t^{\gamma}$ if $f$ is real analytic. This observation tells us that $t^{\gamma}$ power is intrinsic to the Caputo derivative $D_c^{\gamma}$ and the solution is only $\gamma$-H\"{o}lder continuous at $t=0$, but smooth on $(0, T_b)$.
\end{rmk}

In the following theorem, we will show the sign of $f(u(t))$ does not change:
\begin{thm}\label{thm:behavior}
Let $f$ be locally Lipschitz continuous. Suppose $f(u_0)\neq 0$. Then $f(u(t))f(u_0)\ge 0$, $\forall t\in (0, T_b)$, and the equal sign can only be achieved on a nowhere dense set. Consequently, letting $u_c$ be a critical point for \eqref{eq:ode1} in the sense that $f(u_c)=0$ and $f$ changes signs near $u_c$, then all solution curves for \eqref{eq:ode1} do not cross $u=u_c$.
\end{thm}
\begin{proof}
Without loss of generality, we assume $f(u_0)>0$. Define 
\[
t^*=\inf \Big\{ t>0: \exists \delta >0,~ s.t.~f(u(s))\leq 0, ~\forall s \in [t,t+\delta] \Big\}.
\]
First of all, we have that $t^*>0$ since $f(u_0)>0$.
To prove the theorem, we only need to show $t^*=\infty$.

We argue by contradiction. Suppose $t^*<\infty$, then we can find a $\delta >0$, s.t. 
$f(u(t))\leq 0$ for $\forall t \in [t^*,t^*+\delta]$. 
 
We claim that $u(t) < u(t^*)$ for $ \forall t\in (t^*,t^*+\delta) $.
\begin{multline*}
u(t)-u(t^*)=\frac{1}{\Gamma(\gamma)}\left(\int_0^{t}(t-s)^{\gamma-1}f(u(s))ds-\int_0^{t^*}(t^*-s)^{\gamma-1}f(u(s))ds\right)\\
=\frac{1}{\Gamma(\gamma)}\left(\int_0^{t^*}\left((t-s)^{\gamma-1}-(t^*-s)^{\gamma-1}\right) f(u(s))ds+\int_{t^*}^{t}(t-s)^{\gamma-1}f(u(s))ds\right)
\end{multline*}

Notice that $f(u (t))\geq 0$ for $t\in(0,t^*)$,  and $f(u(t))$ is strictly positive when $t$ is close to 0. In addition, $f(u(t))\leq0$ for $t\in [t^*,t^*+\delta]$. Therefore, the right hand side is strictly negative. Hence $u(t) < u(t^*)$ and the claim is proved.

 By the continuity of $u$, we conclude that there exist $t_1, t_2$, $0\le t_1<t_2\leq t^*$, such that for any $s \in [t_1,t_2]$, there exists a $t_s\in [t^*,t^*+\delta], u(s)=u(t_s)$. Then for any $s \in [t_1,t_2]$, $f(u(s))=f(u(t_s))\leq 0$, which contradicts with the definition of $t^*$. 
\end{proof}

For ordinary derivative, as long as we have shown that $f(u(t))$ has a definite sign,  we have that the solution is monotone. For fractional derivatives, this is not obvious, however we can also show that this is true provided $f$ is close to $C^2$. More precisely, we have:

\begin{thm}\label{thm:mon}
Suppose $f\in C^{1}(\alpha, \beta)$ for some interval $(\alpha, \beta)$ and $f'$ is locally Lipschitz on $(\alpha, \beta)$. Then, the solution $u$ to \eqref{eq:ode1} with $u(0)=u_0\in (\alpha, \beta)$ is monotone on  the interval of existence $(0, T_b)$, where $T_b$ is given by Proposition \ref{pro:exist}. If $f(u_0)\neq 0$, the monotonicity is strict.
\end{thm}

Before proving this theorem, let us prove a useful lemma that ensures the positivity of the solution to an integral equation, which is a slightly different version of \cite[Theorem 1]{weis75} (for more discussions on  positivity of solutions to Volterra equations, see \cite{cn79, weis75}) :
\begin{lmm}\label{lmm:pos}
Let $T>0$. Assume $h\in L^1[0,T]$, $h>0~a.e.$, satisfying \[
h(t)-\int_0^tr_{\lambda}(t-s) h(s)ds>0, ~a.e. \forall \lambda>0.
\]
Here $r_{\lambda}$ is the resolvent for kernel $\lambda t^{\gamma-1}$ satisfying 
\begin{gather}\label{eq:resolvent}
r_{\lambda}(t)+\lambda\int_0^t(t-s)^{\gamma-1}r_{\lambda}(s)ds=\lambda t^{\gamma-1}.
\end{gather}
Suppose $v\in C^0[0,T]$, then the integral equation 
\begin{gather}\label{eq:varcoevol}
y(t)+\int_0^t(t-s)^{\gamma-1}v(s) y(s)ds=h(t)
\end{gather}
has a unique solution  $y(t)\in L^1[0, T]$. Further, $
y(t)>0, a.e..
$
In particular, if $h(t)=\alpha t^{\gamma-1}$ for $\alpha>0$, then $y>0$ a.e..
\end{lmm}
\begin{proof}
It can be computed explicitly that \[
r_{\lambda}(t)=-\frac{d}{dt}E_{\gamma}(-\lambda\Gamma(\gamma)t^{\gamma}),
\]
where \[
E_{\gamma}(z)=\sum_{n=0}^{\infty}\frac{z^n}{\Gamma(n\gamma+1)}
\] is the Mittag-Leffler function \cite{mg00,hms11}. $E_{\gamma}(-\lambda\Gamma(\gamma)t^{\gamma})$ is completely monotone that goes from $1$ to $0$ on $(0,\infty)$ \cite{mg00}.  For the concept of completely monotone, see \cite{widder41}.  As a result, $r_{\lambda}\in L^1(0,T)\cap C^0(0,T]$ and $r_{\lambda}>0$. Then, by the fact that the convolution of two locally integrable functions is again locally integrable, all the convolutions are well-defined. (Actually, by an abstract argument, it has also been shown in \cite[Lemma 2.1]{cn79} that $r_{\lambda}$ is  completely monotone and thus non-negative.)

The existence and uniqueness of \eqref{eq:varcoevol} are shown in \cite[Lemma 1]{mf71}. We now prove $y>0,~a.e.$. 

As $v\in C^0[0, T]$, there exists $M>0$ such that $|v|\le M$ on $[0, T]$.  Convolving Equation \eqref{eq:varcoevol} with $r_{\lambda}$, we have
\begin{gather}\label{eq:convtmp}
\int_0^tr_{\lambda}(t-s)y(s)ds+\int_0^t\int_0^{t-s}\lambda(t-s-\tau)^{\gamma-1}r_{\lambda}(\tau)d\tau \frac{v(s)}{\lambda} y(s)ds=\int_0^tr_{\lambda}(t-s)h(s)ds.
\end{gather}
Taking the difference between \eqref{eq:varcoevol} and \eqref{eq:convtmp}, 
\begin{gather*}
y(t)-\int_0^tr_{\lambda}(t-s)y(s)ds+\int_0^tr_{\lambda}(t-s)\frac{v}{\lambda}y(s)ds=h(t)-\int_0^tr_{\lambda}(t-s)h(s)ds
\end{gather*}
where we used the fact $\lambda (t-s)^{\gamma-1}-\lambda\int_0^{t-s}(t-s-\tau)^{\gamma-1}r_{\lambda}(\tau)d\tau=r_{\lambda}(t-s)$ from \eqref{eq:resolvent}.

As a result, $y$ also solves the integral equation
\begin{gather}\label{eq:alterinteq}
y(t)=\left(h(t)-\int_0^tr_{\lambda}(t-s)h(s)ds\right)
+\int_0^t r_{\lambda}(t-s)\left(1-\frac{v(s)}{\lambda}\right) y(s)ds.
\end{gather}
$h(t)-\int_0^tr_{\lambda}(t-s)h(s)ds>0$. Picking $\lambda>M$, $1-\frac{v}{\lambda}>0$ and then $y\ge 0~a.e.$ on $[0, T)$ follows from \cite[Lemma 1]{mf71}. Now from the assumption of this lemma, we have $h(t)-\int_0^tr_{\lambda}(t-s)h(s)ds>0$, in addition we also have $r_{\lambda}>0$ and $1-\frac{v}{\lambda}>0$. As a result, $y>0$ a.e. from \eqref{eq:alterinteq}.

Lastly, if $h(t)=\alpha t^{\gamma-1}$, then \[
h(t)-\int_0^tr_{\lambda}(t-s)h(s)ds=\frac{\alpha}{\lambda} r_{\lambda}(t)>0.
\]
The last claim is proved.
\end{proof}

\begin{rmk}
In the proof of Theorem 1 of \cite{weis75}, the author assumed $h$ to be continuous and the solution to be continuous at $t=0$. In Lemma \ref{lmm:pos}, we do not assume $y$ to be continuous, which is crucial in the case that $h(t)=\alpha t^{\gamma-1}$. 
\end{rmk}

Now, we are able to prove Theorem \ref{thm:mon}:

\begin{proof}[Proof of Theorem \ref{thm:mon}]
Clearly, if $f(u_0)=0$, then $u=u_0$ is the solution by the uniqueness. This is trivially monotone.

Now, we assume $f(u_0)>0$. 
By Lemma \ref{lmm:regu}, $u\in C^1(0,T_b)\cap C^0[0,T_b)$. Now, we fix $T\in (0, T_b)$. The derivative $y=u'$ satisfies the equation \[
y(t)=\frac{f(u_0)}{\Gamma(\gamma)}t^{\gamma-1}+\frac{1}{\Gamma(\gamma)}\int_0^t (t-s)^{\gamma-1}f'(u(s))y(s) ds,~t\in (0, T).
\]
Since $f'(u(t))$ is continuous on $[0, T]$ and $f(u_0)>0$, applying
Lemma \ref{lmm:pos}, we find $y$ is positive on $(0, T)$.
Since $T$ is arbitrary, $y>0$ on $(0, T_b)$. As a result, $u$ is increasing.

If $f(u_0)<0$, we simply consider the equation for $y=-u'$. The argument is similar.
\end{proof}

For usual ODEs, the solution curves do not intersect. For fractional ODEs, we can conclude directly from the integral equation \eqref{eq:vol} that
\begin{pro}\label{pro:nointersect}
If $f(u)$ is locally Lipschitz continuous, non-decreasing, then the solution curves of \eqref{eq:ode1} do not intersect with each other.
\end{pro}

\begin{rmk}\label{rmk:nonintersect}
In the case that $f(u)$ is non-decreasing only on some interval, then as long as one can show that the solutions stay in this interval, then the curves with initial value in this interval does not intersect. For general $f$, it is unclear whether or not the solution curves intersect. The memory is playing a tricky role. 
\end{rmk}

\subsection{Blow-up criterion}

Now we present some results regarding the blow-up behavior. We first have the following observation
\begin{lmm}\label{lmm:toinfty}
Suppose $f(u)$ is locally Lipschitz, non-decreasing on $(0, \infty)$, $u_0>0$ and $f(u_0)>0$. Then, the solution to Equation \eqref{eq:ode1} is non-decreasing on $(0, T_b)$ and $\lim_{t\to T_b^-}u(t)=+\infty$ where $T_b\in (0, \infty]$ is given by Proposition \ref{pro:exist}.
\end{lmm}
\begin{proof}
First of all, let us show that the solution $u$ is non-decreasing on $(0, T_b)$. Note that the $f$ in this lemma is less regular than the function in Theorem \ref{thm:mon}, hence we cannot use Theorem \ref{thm:mon} directly. To show the monotonicity of $u$, let us consider the following sequence of functions $\{u^n\}_{n=0}^{\infty}$:
\[
u^0=u_0, \quad D_c^{\gamma}u^n=f(u^{n-1}),~u^n(0)=u_0,~n\ge 1.
\]
From the integral form of the fractional derivative \eqref{eq:vol}, it is clear that $u^n$ is continuous on $[0,\infty)$. Since $f(u^0)> 0$, we have 
\[
u^1(t)=u_0+\frac{1}{\Gamma(\gamma)}\int_0^t(t-s)^{\gamma-1}f(u^0)ds \ge u_0=u^0(t)
\]
 for $t\in [0,\infty)$. Consequently, $f(u^1(t))\ge f(u^0(t))$ and hence 
 \[
 u^2(t)=u_0+\frac{1}{\Gamma(\gamma)}\int_0^t(t-s)^{\gamma-1}f(u^1(s))ds\ge u_0+\frac{1}{\Gamma(\gamma)}\int_0^t(t-s)^{\gamma-1}f(u^0(s))ds=u^1(t).
 \]
By induction, $u^n(t)\ge u^{n-1}(t)$ for all $n\ge 1$. 

Next, we claim that $u(t)>u^0$ for $t\in (0, T_b)$. For this purpose, we define $t^*=\sup\{\bar{t}\in (0, T_b): f(u(t))>0,~ \forall t\in(0, \bar{t})\}$. We show that $t^*=T_b$. First of all, according to the continuity of $u(t)$ and $f(u)$, and the fact $f(u_0)>0$, we have $t^*>0$. If $t^*<T_b$, then $f(u(t^*))=0$ by the continuity of $f$ and $u$. In addition, by the definition of $t^*$: 
\[
u(t^*)=u^0+\frac{1}{\Gamma(\gamma)}\int_0^{t^*}(t^*-s)^{\gamma-1}f(u(s))ds>u^0.
\]
Since $f$ is non-decreasing, we have $f(u(t^*))\ge f(u^0)>0$, which is a contradiction.

Using $u(t)\ge u^0$, we find 
\[
u(t)=u_0+\frac{1}{\Gamma(\gamma)}\int_0^t(t-s)^{\gamma-1}f(u(s))ds \ge u_0+\frac{1}{\Gamma(\gamma)}\int_0^t(t-s)^{\gamma-1}f(u^0) ds=u^1(t)
\]
 for $t\in [0, T_b)$. Again by induction, $u(t)\ge u^2(t)$ and $u(t)\ge u^3(t)$, etc. Moreover, since $f$ is non-decreasing and $f(u^n)$ is positive, we find that for any $0\le t_1<t_2<\infty$:
 \begin{multline*}
 u^1(t_2)=u_0+\frac{1}{\Gamma(\gamma)}\int_0^{t_2}(t_2-s)^{\gamma-1}f(u^0)ds
 \ge u_0+\frac{1}{\Gamma(\gamma)}\int_{t_2-t_1}^{t_2}(t_2-s)^{\gamma-1}f(u^0)ds\\
=u_0+\frac{1}{\Gamma(\gamma)}\int_{0}^{t_1}(t_1-\tau)^{\gamma-1}f(u^0)d\tau=u^1(t_1).
 \end{multline*}
 The second equality here is achieved by a change of variable $\tau=s-(t_2-t_1)$. 
 Hence, $u^1$ is non-decreasing on $[0,\infty)$. Similarly,
  \begin{multline*}
 u^2(t_2)=u_0+\frac{1}{\Gamma(\gamma)}\int_0^{t_2}(t_2-s)^{\gamma-1}f(u^1(s))ds
 \ge u_0+\frac{1}{\Gamma(\gamma)}\int_{t_2-t_1}^{t_2}(t_2-s)^{\gamma-1}f(u^1(s))ds \\
 \ge u_0+\frac{1}{\Gamma(\gamma)}\int_{t_2-t_1}^{t_2}(t_2-s)^{\gamma-1}f(u^1(s-(t_2-t_1)))ds
=u^2(t_1).
 \end{multline*}
 $u^2$ is non-decreasing on $[0, \infty)$. By induction, $u^n$ is non-decreasing. As a consequence, the sequence $\{u^n(t)\}$ converges to a non-decreasing function $\bar{u}(t)$ for any $t\in [0, T_b)$. By monotone convergence theorem and taking the limit both sides of \[
 u^n(t)=u_0+\frac{1}{\Gamma(\gamma)}\int_0^t(t-s)^{\gamma-1}f(u^{n-1}(s))ds,
 \]
$\bar{u}$ satisfies \eqref{eq:vol}. Thus by the uniqueness of solution (Proposition \ref{pro:exist}) it must be $u$. Hence, $u$ is  non-decreasing.

If $T_b<\infty$, according to the definition of $T_b$ and  the monotonicity, we have $\lim_{t\to T_b^-}u(t)=\infty$. If $T_b=\infty$,  we find
\[
u(t)=u_0+\frac{1}{\Gamma(\gamma)}\int_0^t(t-s)^{\gamma-1}f(u(s))ds
\ge u_0+\frac{f(u_0)}{\Gamma(\gamma)}\int_0^t(t-s)^{\gamma-1}ds\to\infty.
\]
\end{proof}

The next result, which is an Osgood type criterion is essentially from \cite{by12} for the Volterra type integral equations. Here, we reinterpret it for our fractional ODE \eqref{eq:ode1}, and using similar ideas we present an improved proof, which enables us to improve the bounds of blow-up time in Section \ref{sec:specialcase}:
\begin{pro}\label{pro:blow-upcriteria}
Suppose $f(u)$ is locally Lipschitz, non-decreasing on $(0,\infty)$, $u_0>0$ and $f(u_0)>0$. Then, $T_b<\infty$ if and only if there exists $U>0$
\begin{gather}
\int_U^{\infty}\left(\frac{u}{f(u)}\right)^{1/\gamma}\frac{du}{u}<\infty.
\end{gather}
\end{pro}
\begin{proof}
Consider the equivalent Volterra type equation \eqref{eq:vol}. By Lemma \ref{lmm:toinfty}, $u$ is increasing and $u(t)\to \infty$ as $t\to T_b^-$. Pick $r>\max(1, u_0^{1/\gamma})$. There exists $t_n<T_b$ so that $u(t_n)=r^{n\gamma}$ for $n=1,2,\ldots$. 

By \eqref{eq:vol}, we have 
\begin{multline*}
u(t_n)=u_0+\frac{1}{\Gamma(\gamma)}\int_0^{t_n}(t_n-s)^{\gamma-1}f(u(s))ds
\ge \frac{f(u(t_{n-1}))}{\Gamma(\gamma)}\int_{t_{n-1}}^{t_n}(t_n-s)^{\gamma-1}ds\\
=\frac{1}{\Gamma(1+\gamma)}f(u(t_{n-1})) (t_n-t_{n-1})^{\gamma}.
\end{multline*}
As a result, there exist constants $C(\gamma)>0$ and $C_1(\gamma, r)>0$ such that
 \[
t_n-t_{n-1}\le C(\gamma)\left(\frac{u(t_n)}{f(u(t_{n-1}))}\right)^{1/\gamma}
=C(\gamma)\frac{r^2}{r-1}\frac{r^{n-1}-r^{n-2}}{f(r^{(n-1)\gamma})^{1/\gamma}}
\le C_1(\gamma,r)\int_{r^{n-2}}^{r^{n-1}}\frac{1}{f(\tau^{\gamma})^{1/\gamma}}d\tau.
\]
On the other hand, 
\begin{multline*}
u(t_n)=u_0+\frac{1}{\Gamma(\gamma)}\int_0^{t_n}(t_n-s)^{\gamma-1}f(u(s))ds \\
\le u_0+\frac{1}{\Gamma(\gamma)}\int_0^{t_{n-1}}(t_{n-1}-s)^{\gamma-1}f(u(s))ds+\frac{1}{\Gamma(\gamma)}\int_{t_{n-1}}^{t_{n}}(t_n-s)^{\gamma-1}f(u(s))ds\\
\le u(t_{n-1})+\frac{f(u(t_n))}{\Gamma(1+\gamma)}(t_n-t_{n-1})^{\gamma}.
\end{multline*}
As a result, there exist two constants $\bar{C}_1(\gamma)>0$ and $\bar{C}_2(\gamma, r)>0$ such that
\[
t_n-t_{n-1} \ge \bar{C}_1(\gamma)\frac{(r-1)^{1/\gamma}}{r(r-1)}\frac{r^{n+1}-r^{n}}{f(r^{n\gamma})^{1/\gamma}}\ge \bar{C}_2(\gamma, r)\int_{r^n}^{r^{n+1}}
\frac{1}{f(\tau^{\gamma})^{1/\gamma}}d\tau.
\]
Hence, $T_b<\infty$ if and only if $\int^{\infty}\frac{1}{f(\tau^{\gamma})^{1/\gamma}}d\tau<\infty$, or there exists some $U>0$ such that
\[
\int_U^{\infty}\left(\frac{u}{f(u)}\right)^{1/\gamma}\frac{du}{u}<\infty.
\]
\end{proof}

\section{Comparison principles}\label{sec:comp}
The following comparison principle (\cite[Theorem 7]{liliu16}) is useful when we study the behavior of \eqref{eq:ode1} and derive certain Gr\"{o}nwall type inequalities:
\begin{pro}[\cite{liliu16}]\label{pro:comp}
Suppose $f(u)$ is locally Lipschitz, non-decreasing on some interval $(\alpha, \beta)$. Suppose $v_1: [0, T)\mapsto (\alpha, \beta)$ is continuous. If $v_1$ satisfies $$
D_c^{\gamma}v_1\le f(v_1), ~on~[0,T),
$$
where this inequality means $D_c^{\gamma}v_1-f(v_1)$ is a nonpositive distribution (see \cite[Def. 6]{liliu16}). Let $v_2$ be the unique solution to the equation $$
D_c^{\gamma}v_2=f(v_2), \ \ v_2(0)\in (\alpha, \beta),
$$
on $[0, T_b)$ provided by Proposition \ref{pro:exist}. If $v_2(0)\ge v_1(0)$, then on $[0, \min(T,T_b))$, $v_1(t)\le v_2(t)$.

Correspondingly, if $v_1$ satisfies 
$$
D_c^{\gamma}v_1\ge f(v_1), ~on~[0,T),
$$
where the inequality means $D_c^{\gamma}v_1-f(v_1)$ is a nonnegative distribution and $v_2$ is the solution to $$
D_c^{\gamma}v_2=f(v_2), \ \ v_2(0)\in (\alpha, \beta).
$$
If $v_2(0)\le v_1(0)$, then $v_1(t)\ge v_2(t)$ on $[0, \min(T, T_b))$.
\end{pro}

Using the idea of the proof for \cite[Theorem 7]{liliu16}, we are able to show some other versions of comparison principles. For example, an integral version is as follows
\begin{pro}
Suppose $f(\cdot)$ is locally Lipschitz, non-decreasing on $(\alpha, \beta)$. If a continuous function $v: [0, T)\mapsto (\alpha, \beta)$ satisfies the following inequality
\[
v(t)\le u_0+\frac{1}{\Gamma(\gamma)}\int_0^t(t-s)^{\gamma-1}f(v(s))ds,
~~t\in [0, T),
\]
where $u$ is the solution of \eqref{eq:ode1} with initial value $u_0\in (\alpha, \beta)$ on $[0, T_b)$, then we have $v\le u$ on $(0, \min(T, T_b))$. 

Similarly, if 
\[
v(t)\ge u_0+\frac{1}{\Gamma(\gamma)}\int_0^t(t-s)^{\gamma-1}f(v(s))ds,
~~t\in [0, T),
\]
then we have $v\ge u$ on $(0, \min(T, T_b))$.
\end{pro}
Note that the integral version is not a pure repetition of Proposition \ref{pro:comp} since we do not necessarily have $D_c^{\gamma}v\le f(v)$ ($D_c^{\gamma}v\ge f(v)$) for all $t\in (0, T)$. 
Another version of comparison principle is a corollary of Proposition \ref{pro:comp}:
\begin{cor}\label{cor:g}
Suppose both $f_1(u)$ and $f_2(u)$ are locally Lipschitz on $(\alpha, \beta)$, satisfying $f_1(u)\ge f_2(u)$ for any $u\in (\alpha, \beta)$. Assume that one of them is non-decreasing. Let $u_1$ and $u_2$ be the solutions of $D_c^{\gamma}u=f_1(u)$ and $D_c^{\gamma}u=f_2(u)$ on the intervals $(0,T_b^1)$ and $(0,T_b^2)$ with initial values $u_1(0)$ and $u_2(0)$ respectively. If in addition $\alpha<u_2(0)\le u_1(0)<\beta$, then \[
u_1\ge u_2, \forall t\in (0, \min(T_b^1, T_b^2)).
\]
\end{cor}
\begin{proof}
First, assume that $f_2$ is non-decreasing. Then, we have \[
D_c^{\gamma}u_1=f_1(u_1)\ge f_2(u_1).
\]
Applying Proposition \ref{pro:comp} for $D_c^{\gamma}u_1\ge f_2(u_1)$ yields the claim.

If for otherwise $f_1$ is non-decreasing, we have \[
D_c^{\gamma}u_2=f_2(u_2)\le f_1(u_2).
\]
Applying Proposition \ref{pro:comp} for $D_c^{\gamma}u_2\le f_1(u_2)$ yields the claim.
\end{proof}

\section{blowup and long time behavior for a class of fractional ODEs} \label{sec:specialcase}

In this section, we will focus on the cases $f(u)=Au^p$ and $(\alpha, \beta)=(0, \infty)$ for simplicity.
\begin{gather}\label{eq:ode2}
D_c^{\gamma}u=Au^p,~ u(0)=u_0>0.
\end{gather}
This type of equations are general enough. For example, if $p>0$ and there exist $C_1>0$, $C_2 >0$ such that $C_1u^p\le f(u)\le C_2u^p$, then the solution is under control according to Corollary \ref{cor:g}. We will discuss in different cases to show that \eqref{eq:ode2} shares the regularity properties of normal time derivative ODE. Moreover, one can also prove some ``time-delay properties'' of \eqref{eq:ode2}.

\subsection{Finite time blowup}
As an application of Proposition \ref{pro:blow-upcriteria}, we have the following theorem:
\begin{thm}\label{thm:increasing}
Let $p\ge 0$, $A>0, u_0>0$, and $u(t)$ be the unique solution to fractional ODE \eqref{eq:ode2}. Then, we have the following claims: (1). $u(t)$ is an increasing function and $\lim_{t\to T_b^-}u(t)=\infty$. (2). All the solution curves with $u_0\ge 0$ do not intersect with each other. (3). If $0\le p\le 1$,  $T_b=\infty$, i.e. the solution exists globally.  (4). If $p>1$, $u(t)$ blows up in finite time (i.e.  $T_b<\infty$ and $\lim_{t\to T_b^-}u(t)=\infty$).
\end{thm}
\begin{proof}
By Lemma \ref{lmm:toinfty}, $u(\cdot)$ is an increasing function and $\lim_{t\to T_b^-}u(t)=\infty$. Since $u\ge u_0>0$, and $Au^p$ is increasing in $(0,\infty)$, by Proposition \ref{pro:nointersect} and Remark \ref{rmk:nonintersect}, the solution curves wtih $u_0\ge 0$ do not intersect.

On $[u_0,\infty)$, $f(u)=Au^p$ is locally Lipschitz continuous. As a corollary of Proposition \ref{pro:blow-upcriteria},  when $0\le p\le 1$, the solution exists globally, i.e $T_b=\infty$. 
And when $p>1$, the solution blows up in finite time.
\end{proof}

\subsection{The bounds of blow-up time} \label{subsec:bounds}
In this subsection, our main goal is to find suitable bounds of the blow-up time and to understand the effects of the memory introduced by the Caputo derivatives. Clearly, one possible lower bound is the radius of convergence of the power series $u=\sum_{n=0}^{\infty}a_n t^{n\gamma}$, however the asymptotic behavior of $a_n$ is hard to find. In \cite{roberts1993volterra}, the author provided some bounds for the blow-up time of the integral equation \eqref{eq:vol}. In this paper, we have the following improved result:

\begin{pro}\label{pro:bound}
Suppose $\gamma\in (0,1)$, $p>1$, $A>0$, and $u_0>0$. Let $T_b$ be the blow-up time of the solution to \eqref{eq:ode2}. Then, we have the following inequality
\begin{multline} 
\left(\frac{\Gamma(1+\gamma)}{Au_0^{p-1}}\right)^{1/\gamma}
\sup_{r>1}\frac{(r^{\gamma}-1)^{1/\gamma}}{r(r^{p-1}-1)} \le T_b \\
\le 
\left(\frac{\Gamma(1+\gamma)}{Au_0^{p-1}}\right)^{1/\gamma}\inf_{r>1, m\in\mathbb{Z}_+}\left(\frac{r^p}{r^{(m+1)(p-1)}-r^{m(p-1)}}+\left(\frac{1-r^{m\gamma(1-p)}}{p-1}\right)^{1/\gamma}\right).
\end{multline}
\end{pro}

\begin{proof}
Let $r>1$. We now choose $t_n$ such that 
\[
u(t_n)=u_0 r^{n\gamma}.
\]
It is then clear that $0=t_0<t_1<t_2\ldots$. For convenience, we denote
\[
k(t)=\frac{1}{\Gamma(\gamma)}t^{\gamma-1}, ~~K(t)=\frac{1}{\Gamma(1+\gamma)}t^{\gamma}.
\]
The following relation
\begin{multline*}
u(t_n)=u_0+\int_0^{t_{n-1}}k(t_n-s)f(u(s))ds+\int_{t_{n-1}}^{t_n}k(t_n-s)f(u(s))ds\\
\le u_0+\int_0^{t_{n-1}}k(t_{n-1}-s)f(u(s))ds+K(t_n-t_{n-1})f(u(t_n))
\end{multline*}
yields that
\[
K(t_n-t_{n-1})\ge \frac{u_0r^{n\gamma}(1-r^{-\gamma})}{f(u_0 r^{n\gamma})}.
\]
Hence
 \begin{equation}\label{grow}
 t_n-t_{n-1}\ge \left(\frac{\Gamma(1+\gamma)}{Au_0^{p-1}}\right)^{1/\gamma}\frac{(1-r^{-\gamma})^{1/\gamma}}{ r^{n(p-1)}}.
\end{equation}
As a result,
\[
T_b=\sum_{n=1}^{\infty}t_n-t_{n-1} \ge  \left(\frac{\Gamma(1+\gamma)}{Au_0^{p-1}}\right)^{1/\gamma} \frac{(r^{\gamma}-1)^{1/\gamma}}{r(r^{p-1}-1)}.
\]

To prove the upper bound, we fix $m\ge 1$, and then find
\begin{gather}\label{eq:upperest}
u(t)\ge u_0+\int_{0}^{t}\frac{1}{\Gamma(\gamma)}(t_m-s)^{\gamma-1} f(u(s))ds:=v(t), ~t\in (0, t_m).
\end{gather}
It is clear that $v(t_m)=u(t_m)$. As a result, \[
v'(t)=\frac{1}{\Gamma(\gamma)}(t_m-t)^{\gamma-1}f(u(t))\ge \frac{1}{\Gamma(\gamma)}(t_m-t)^{\gamma-1}f(v(t))
\]
and \[
\int_{u_0}^{u(t_m)}\frac{dv}{f(v)} \ge \frac{1}{\Gamma(1+\gamma)}t_m^{\gamma},
\]
implying 
\begin{gather}\label{eq:t1}
t_m\le \left(\frac{\Gamma(1+\gamma)}{A(p-1)u_0^{p-1}}\right)^{1/\gamma}(1-r^{m\gamma(1-p)})^{1/\gamma}.
\end{gather}

For $n\ge m+1$, we find
\begin{gather*}
u(t_n)\ge u_0+\int_{t_{n-1}}^{t_n}\frac{1}{\Gamma(\gamma)}(t_n-s)^{\gamma-1} f(u(t_{n-1}))ds,
\end{gather*}
and thus 
\begin{gather}\label{eq:tn}
 \frac{1}{\Gamma(1+\gamma)}(t_n-t_{n-1})^{\gamma} f(u(t_{n-1}))\le u_0r^{n\gamma}-u_0\le u_0r^{n\gamma}.
\end{gather}
Combining \eqref{eq:t1} and \eqref{eq:tn}, we finally have the upper bound, 
\begin{multline*}
T_b=\sum_{n=m+1}^{\infty}(t_n-t_{n-1})+t_m\le \left(\frac{\Gamma(1+\gamma)}{Au_0^{p-1}}\right)^{1/\gamma}
\frac{r^p}{r^{(m+1)(p-1)}-r^{m(p-1)}}+t_m
\\
= \left(\frac{\Gamma(1+\gamma)}{Au_0^{p-1}}\right)^{1/\gamma}
\left(\frac{r^p}{r^{(m+1)(p-1)}-r^{m(p-1)}}+\left(\frac{1-r^{m\gamma(1-p)}}{p-1}\right)^{1/\gamma}\right).
\end{multline*}
\end{proof}

In the proof of the upper bound, the estimate we did for $t_m$ essentially follows the method in \cite{roberts1993volterra}.
By optimizing the constants we get in the Proposition \ref{pro:bound}, we have
\begin{thm}\label{buT}
Let $\gamma \in (0,1)$, $p>1$, $A>0$ and $u_0>0$. The following bounds for the blow-up time $T_b$ of Equation \eqref{eq:ode2} hold, 
\begin{gather}
\left(\frac{\Gamma(1+\gamma)}{Au_0^{p-1}G(p)}\right)^{1/\gamma}\le T_b \le 
\left(\frac{\Gamma(1+\gamma)}{Au_0^{p-1}H(p,\gamma)}\right)^{1/\gamma} ,
\end{gather}
where
\begin{gather}\label{eq:blowconst}
G(p)=\min\left(2^p, \frac{p^p}{(p-1)^{p-1}}\right), ~ H(p, \gamma)=\max\left(p-1, 2^{-\frac{p\gamma}{p-1}}\right).
\end{gather}

Consequently, with $A>0, p>1$ fixed, there exist $u_{02}>u_{01}>0$ such that whenever $u_0<u_{01}$, $\lim_{\gamma\to 0^+}T_b=\infty$, while $u_0>u_{02}$ implies $\lim_{\gamma\to 0^+}T_b=0$.
\end{thm}

\begin{proof}
For the lower bound, picking $r=2^{1/\gamma}>1$, we find 
\[
 \sup_{r>1}\frac{(r^{\gamma}-1)^{1/\gamma}}{r^p-r}\ge \frac{1^{1/\gamma}}{2^{p/\gamma}-2^{1/\gamma}}\ge \frac{1}{2^{p/\gamma}}.
\]
Similarly, picking $r=(p/(p-1))^{1/\gamma}$ yields\[
\sup_{r>1}\frac{(r^{\gamma}-1)^{1/\gamma}}{r^p-r}\ge \left(\frac{(p-1)^{p-1}}{p^p}\right)^{1/\gamma}.
\]

For the upper bound, we fix $m>\frac{1}{p-1}$, and let $r\to\infty$:
\[
\frac{r^p}{r^{(m+1)(p-1)}-r^{m(p-1)}}+\left(\frac{1-r^{m\gamma(1-p)}}{p-1}\right)^{1/\gamma}\to \left(\frac{1}{p-1}\right)^{1/\gamma}.
\] 
If instead we choose $m=1$ and $r=2^{1/(p-1)}>1$, we have 
\[
\frac{r^p}{r^{(m+1)(p-1)}-r^{m(p-1)}}+\left(\frac{1-r^{m\gamma(1-p)}}{p-1}\right)^{1/\gamma}
=\frac{1}{2}2^{\frac{p}{p-1}}+\frac{1}{2}\left(\frac{2^{\gamma}-1}{p-1}\right)^{1/\gamma}.
\]
Consider $Q(p):=2^{\frac{p}{p-1}}\left(\frac{p-1}{2^{\gamma}-1}\right)^{1/\gamma}$. By elementary calculus, we have
\[
Q'(p)=Q(p)\left(\log Q(p)\right)'=Q(p)\frac{1}{(p-1)^2}\frac{1}{\gamma}(p-1-\gamma\log 2).
\]
 Hence, \[
Q(p)\ge Q(\gamma\log 2+1)
=2\left(\frac{\gamma e\log 2}{2^{\gamma}-1}\right)^{1/\gamma}\ge 2(e\log 2)^{1/\gamma}\ge 2.
\]
For the second inequality, note that $\gamma-2^{\gamma}+1$ is concave on $(0,1)$ and equals zero at $\gamma=0,1$, so $\gamma>2^{\gamma}-1$ for $\gamma\in (0,1)$.
We find 
\[
T_b\le \frac{1}{2}2^{\frac{p}{p-1}}+\frac{1}{2}\left(\frac{2^{\gamma}-1}{p-1}\right)^{1/\gamma}\le \frac{3}{4}2^{\frac{p}{p-1}}<2^{\frac{p}{p-1}},
\]
and the upper bound follows.

As long as we have these two bounds, it is clear that we can pick 
\[
u_{01}=A^{-\frac{1}{p-1}}\max\left(2^{-\frac{p}{p-1}}, \frac{p-1}{p^{p/(p-1)}}\right),~u_{02}=A^{-\frac{1}{p-1}}\min\left(1, \Big(\frac{1}{p-1}\Big)^{1/(p-1)}\right).
\]
\end{proof}

\begin{rmk}
From Theorem \ref{buT}, one can clearly see how the memory plays the role. The memory is getting stronger as $\gamma$ goes closer to $0$. When $u_0$ is very small, the strong memory defers the blowup. If $u_0$ is large, the strong memory accelerates the blowup. For the critical value of $u_0$, we believe it is determined by the limiting case $\gamma=0$:
\[
D_c^{\gamma}u=u-u_0=Au^p.
\]
If $u_0>\frac{p-1}{p}\left(\frac{1}{pA}\right)^{1/(p-1)}$, this algebraic equation has no solution and it means the blow-up time is zero.
If $u_0<\frac{p-1}{p}\left(\frac{1}{pA}\right)^{1/(p-1)}$, there is a constant solution for $t>0$ which means the blow-up time is infinity.
\end{rmk}
\begin{rmk}
The estimates $\frac{(p-1)^{p-1}}{p^p}$ and $(\frac{1}{p-1})^{1/\gamma}$ for the blow-up time can also be obtained by the results in \cite{roberts1996growth} for the Volterra integral equations, but we have better constants here. One may observe the following two facts:
\begin{itemize}
\item $p\geq 2$ if and only if $\frac{(p-1)^{p-1}}{p^p}\geq \frac{1}{2^p}$. Hence for $p\in (1,2)$, $G(p)=2^p$ in \eqref{eq:blowconst} while for $p\in (2,\infty)$, $G(p)=\frac{p^p}{(p-1)^{p-1}}$. The latter gives asymptotic behavior for large $p$. 
\item For the upper bound, if $p<2$, as $\gamma$ is small enough,  $2^{-\frac{\gamma p}{p-1}}>p-1$ and $H(p, \gamma)=2^{-\frac{\gamma p}{p-1}}$  in \eqref{eq:blowconst} while for large $p$, $H(p,\gamma)=p-1$ and it gives the asymptotic bound for large $p$.
\end{itemize}
\end{rmk}

\begin{rmk}\label{rate}
One may wonder the asymptotic behavior, or so-called growth rate of the solution near the blow-up time. There are a lot of references about this topic. One can check, for instance \cite{roberts1993volterra,roberts1996growth}. To find the correct power of the blow-up profile, one can plug $\frac{1}{(T-t)^{\alpha}}$ into \eqref{eq:ode2} and use the heuristic calculation $D_c^{\gamma}(\frac{1}{(T-t)^{\alpha}})\approx \frac{1}{(T-t)^{\alpha+\gamma}}$, which means $\alpha+\gamma=p\alpha$, or $\alpha=\frac{\gamma}{p-1}$. In fact, from (3.2) in \cite{roberts1996growth}, the solution to \eqref{eq:ode2} satisfies
\begin{equation}\label{ratea}
u(t)-u_0\sim \left[\frac{\Gamma(\frac{p\gamma}{p-1})}{A\Gamma(\frac{\gamma}{p-1})}\right]^{\frac{1}{p-1}}\left((T_b-t)^{-1}-T_b^{-1}\right)^{\frac{\gamma}{p-1}},\mbox{ as }t\rightarrow T_b^-. 
\end{equation}
One can find the proof in the appendix (Section \ref{sec:app}). In addition, as in \cite{roberts1993volterra,roberts1996growth}, one can expect explicit asymptotic behavior for more general $f(u)$. 
\end{rmk}

\subsection{Other cases}
In this subsection, we discuss other choices of the parameters $A$ and $p$ in \eqref{eq:ode2}.

First of all, we investigate the cases $A>0$ and $p<0$.
\begin{thm}
Let $A>0, p< 0$ and $u_0>0$, and $u(t)$ be the solution to \eqref{eq:ode2}. Then, $u$ exists globally on $(0,\infty)$ and is increasing. Moreover, 
\[
u_0\le u(t)\le u_0+Au_0^{p}\frac{1}{\Gamma(1+\gamma)}t^{\gamma}.
\]
\end{thm}

\begin{proof}
We define
\[
\tilde{f}(u)=
\begin{cases}
Au^p, & u\ge u_0,\\
Au_0^p, & u<u_0.
\end{cases}
\] 
Consequently, $\tilde{f}$ is locally Lipschitz. 

By Proposition \ref{pro:exist}, $D_c^{\gamma}v=\tilde{f}(v)$ has a unique solution $v$ with an interval of existence $[0, T_b)$. Clearly, $\tilde{f}(v(t))>0$ on $[0, T_b)$, and consequently $v(t)>u_0$ for $t\in (0, T_b)$. This implies that $v$ is actually the solution to $D_c^{\gamma}u=Au^p$ on $[0, T_b)$. We therefore identify $v$ with $u$.  The monotonicity of $u$ follows from Theorem \ref{thm:mon}. 

We compare $u$ with the solution of $D_c^{\gamma}w=Au_0^{p}, ~w(0)=u_0$ using Corollary \ref{cor:g} and find that
\[
u(t)\le w(t)=u_0+Au_0^{p}\frac{1}{\Gamma(1+\gamma)}t^{\gamma}.
\]
This implies that $T_b=\infty$.
\end{proof}

We now consider $A<0$. 
\begin{pro}\label{pro:negA}
Let $A<0$, $u_0>0$ and $p\in \mathbb{R}$. There exists $T_b\in (0, \infty]$, such that \eqref{eq:ode2} has a unique solution $u$ on $(0, T_b)$. Moreover, $0<u(t)< u_0$ and it is decreasing on $(0, T_b)$, satisfying
\begin{gather}
\lim_{t\to T_b^-}u(t)=0.
\end{gather}
\end{pro}

\begin{proof}
Since since $Au^p$ ($p\in\mathbb{R}$) is locally Lipschitz on $(0, \infty)$, applying Proposition \ref{pro:exist} for the interval $(\alpha, \beta)=(0, \infty)$,  \eqref{eq:ode2} has a unique solution $u$ with $u(0)=u_0>0$ on $(0, T_b)$.

Since $f(u)=Au^p$ ($A<0, p\in\mathbb{R}$) is smooth on $(0,\infty)$ and $f(u_0)\neq 0$, by Theorem \ref{thm:mon}, $u$ is strictly monotone. Using the integral form $\eqref{eq:vol}$, it is clear that $u(t)< u_0$ for $t>0$. Hence $u$ is decreasing. From Proposition \ref{pro:exist}, either $T_b=\infty$ or $T_b<\infty$ and $\lim_{t\to T_b^-}u(t)=0$. To finish the proof, we only need to show that if $T_b=\infty$, $\lim_{t\to T_b}u(t)=0$. Suppose for otherwise $\lim_{t\to T_b}u(t)\neq 0$. Then, $u(t)$ is bounded below by $\delta>0$. 
Then, as $t\to T_b=\infty$, \[
u(t)=u_0-\frac{1}{\Gamma(\gamma)}\int_0^{t}|A| u(s)^p(t-s)^{\gamma-1}ds\le u_0-\frac{1}{\Gamma(\gamma)}\int_0^{t}|A|\min( \delta^p,u_0^p)(t-s)^{\gamma-1}ds  \to -\infty,
\]
which is a contradiction.
\end{proof}

\begin{rmk}
In the case $p<1$, it is possible that $T_b<\infty$ and $Au^p$ is defined on $\mathbb{R}$. The solution may be extended beyond $T_b$. However, $Au^p$ may not be Lipschitz continuous at $u=0$ and it makes the analysis complicated (of course $p=0$ case is trivial and we have $u(t)=u_0+Ag_{\gamma+1}$).
\end{rmk}

In the case $p<1$, $Au^p$ may not be Lipschitz at $u=0$. Hence, for simplicity, we only consider $A<0, p\ge 1$ for further discussion.   Actually, we are able to show:
\begin{thm}\label{thm:slowdecay}
Fix $\gamma\in (0,1)$, $p\ge 1$, $u_0>0$ and $A<0$. Let $u(t)$ be the unique solution to \eqref{eq:ode2} with initial value $u_0$. Then, $u(t)>0, \forall t>0$. $u(\cdot)$ is decreasing and $\lim_{t\to\infty}u(t)=0$. Moreover, there exists $C(u_0,A, p)>0$ such that when $t$ is large enough,
\begin{gather}\label{eq:slowdecay}
u(t)\ge C(u_0,A,p)\frac{t^{-\gamma/p}}{\Gamma(1-\gamma)^{1/p}}.
\end{gather}
\end{thm}

\begin{proof}
First of all, by Proposition \ref{pro:negA}, $u$ is decreasing. Pick $r\in (0,1)$. 
 By the fact $\lim_{t\to T_b^-}u(t)=0$, we are able to pick disjoint intervals $J_n=(t_{n-1}, t_n)$ such that $u$ stays between $u(t_{n-1})$ and $u(t_n)$ inside $J_n$ and $u(t_m)=u_0r^{m\gamma}$ . 
 
Therefore, 
\begin{multline*}
u(t_n)=u_0-\frac{1}{\Gamma(\gamma)}\int_0^{t_{n-1}}(t_n-s)^{\gamma-1}|A|u(s)^pds-\frac{1}{\Gamma(\gamma)}\int_{t_{n-1}}^{t_n}
(t_n-s)^{\gamma-1}|A|u(s)^pds\\
\ge u(t_{n-1})-C_1(u(t_{n-1}))^p |J_n|^{\gamma}.
\end{multline*}
This implies that \[
|J_n|\ge C_2(r, u_0, p,\gamma) r^{-(n-1)(p-1)}.
\]
As a result, \[
T_b\ge \sum_{n} |J_n|=\infty.
\]
It follows that for any $t>0$, $u(t)>0$, $u$ is decreasing and \[
\lim_{t\to\infty}u(t)=0.
\]

Since $u\in C^1(0, T_b)\cap C^0[0, T_b)$,  integrating by parts from \eqref{eq:captra} gives us the alternative expression for Caputo derivative \eqref{eq:capint}. Now, $u(t)\le u(s)$ for all $s\le t$ and as a result,
 \[
\Gamma(1-\gamma)D_c^{\gamma}u(t) \le (u(t)-u(0))t^{-\gamma}.
\]
When $t$ is large enough,
\[
u^p(t)\ge \frac{u_0}{2A}\frac{t^{-\gamma}}{\Gamma(1-\gamma)}.
\]
\end{proof}

\begin{rmk}
The proof of Theorem \ref{thm:slowdecay} is quite indirect.  The equation may be rewritten as \[
u(t)+|A| g_{\gamma}*(\theta(t)u^p)=u_0
\]
and $|A|u^p$ is an $m$-accretive operator (see \cite{cn81}) of $u$ when $p>1, u>0$. This form is related to the equations studied in \cite{cn81} and may yield some direct proof using functional analysis. In the case that the kernel is not $L^1$, \cite{cn81} requires that $m$-accretive operator to be coercive which does not apply here. 
\end{rmk}

\begin{rmk}
It is well known that $\gamma=1$ yields $u(t)\sim Ct^{-1/(p-1)}$, which decays to zero faster than $t^{-\gamma/p}$. The memory really gives a slow decaying rate.  As $\gamma\to 1$, $\Gamma(1-\gamma)\to\infty$ and the dominant term in \eqref{eq:slowdecay} vanishes. This means $t^{-1/(p-1)}$ must appear in the next order and the slow decaying dominate term \eqref{eq:slowdecay} is an effect of memory. 
\end{rmk}

\begin{rmk}
Regarding the asymptotic behavior of Caputo derivative, we may consider the derivative of $(1+t)^p$. If $p>0$, \[
D_c^{\gamma}(1+t)^p \sim Cp t^{p-\gamma}, ~ t\to\infty
\]
since $(1+t)^p$ is smooth and one can use \eqref{eq:captra} to compute. In the decaying cases $p<0$, \[
D_c^{\gamma}(1+t)^p\sim Cp t^{-\gamma}, ~ t\to\infty.
\]
This means no matter how fast the function decays, the Caputo derivative is always like $-Ct^{-\gamma}$ asymptotically, which can also be confirmed through the proof of Theorem \ref{thm:slowdecay}. 

Actually, $-t^{-\gamma}$ should be the intrinsic rate for the Caputo derivative of decaying functions. If, for example, $D_c^{\gamma}u(t) \le -C(1+t)^{-\gamma+\delta}$ for some $C>0$ and $\delta>0$, then $u(t)\to-\infty$. 
Conversely, if $D_c^{\gamma}u(t) \sim -(1+t)^{-\gamma-\delta}$, then $u$, though is less than $u_0$, will eventually go back to $u_0$. Notice that though the Caputo derivative is negative, the function does not always decay. This is because the decaying property at the earlier stage lingers to later stage due to memory.
\end{rmk}

\section{Discrete equations and numerical simulations}\label{sec:discrete}

In this section, we study discrete equations obtained from discretizing the differential equation  \eqref{eq:ode1} or the integral equation \eqref{eq:vol}. We will consider some typical numerical schemes which are useful in different situations (e.g. stability analysis for numerical schemes or the proof of existence of weak solutions to fractional PDEs). 

In this section,  $k>0$ is the time step, and $t_n=nk$. $u^n$ is the computed numerical value at $t_n$ and $u(t_n)$ is the value of the solution to \eqref{eq:ode1} evaluated at $t_n$. 

\subsection{Schemes for the integral equation}

Consider discretizing \eqref{eq:vol} with explicit schemes. We have
\begin{gather}\label{eq:schemeInt}
u^{n}=u_0+\frac{k^{\gamma}}{\Gamma(1+\gamma)}\sum_{m=0}^{n-1} f(u^{m})\Big((n-m)^{\gamma}-(n-m-1)^{\gamma}\Big)
\end{gather}
or
\begin{gather}\label{eq:schemeInt2}
u^{n}=u_0+\frac{k^{\gamma}}{\Gamma(\gamma)}\sum_{m=0}^{n-1} f(u^{m})(n-m)^{\gamma-1}.
\end{gather}

To study these two schemes, we first prove the following discrete Gr\"{o}nwall inequalities:
\begin{lmm}\label{lmm:discretegronwall1}
Let $f(u)$ be nonnegative, non-decreasing, locally Lipschitz on $[0,\infty)$ and let $u_0>0$. 
Suppose $w^n$ ($0\le n\le N$) is a nonnegative sequence ($w^n\ge 0$) such that 
\[
w^{n}\le u_0+\frac{k^{\gamma}}{\Gamma(1+\gamma)}\sum_{m=0}^{n-1} f(w^{m})\Big((n-m)^{\gamma}-(n-m-1)^{\gamma}\Big),
~ 0\le n\le N,
\] 
or 
\[
w^{n}\le u_0+\frac{k^{\gamma}}{\Gamma(\gamma)}\sum_{m=0}^{n-1} f(w^{m})(n-m)^{\gamma-1},~  0\le n\le N,
\]
then \[
w^n \le u(nk),~ 0\le n\le N,~n<T_b/k,
\]
where $u(t)$ is the unique solution to the fractional ODE \eqref{eq:ode1} with initial condition $u(0)=u_0$.
\end{lmm}

\begin{proof}
We prove by induction. $n=0$ is clearly true. Now, let $1\le n \le N$ and assume that $w^m\le u(mk)$ for all $m\le n-1$. Then, by the non-decreasing property of $f$ and the induction assumption, we have
\begin{multline*}
w^{n}\le u_0+\frac{k^{\gamma}}{\Gamma(1+\gamma)}\sum_{m=0}^{n-1} f(w^{m})\Big((n-m)^{\gamma}-(n-m-1)^{\gamma}\Big)\\
\le u_0+\frac{k^{\gamma}}{\Gamma(1+\gamma)}\sum_{m=0}^{n-1} f(u(mk))\Big((n-m)^{\gamma}-(n-m-1)^{\gamma}\Big)\\
=u_0+\frac{1}{\Gamma(\gamma)}\sum_{m=0}^{n-1} \int_{t_{m}}^{t_{m+1}}
(nk-s)^{\gamma-1}f(u(mk))ds.
\end{multline*} 
Since $f(\cdot)$ is non-negative, by Lemma \ref{lmm:toinfty}, $u(\cdot)$ is increasing. As a result,
\begin{multline*}
w^{n}\le u_0+\frac{1}{\Gamma(\gamma)}\sum_{m=0}^{n-1} \int_{t_{m}}^{t_{m+1}}
(nk-s)^{\gamma-1}f(u(mk))ds
\le u_0+\frac{1}{\Gamma(\gamma)}\int_0^{nk}(nk-s)^{\gamma-1}f(u(s))ds=u(t^n).
\end{multline*}

The proof for the other inequality is similar due to the fact \[
\frac{k^{\gamma}}{\Gamma(\gamma)}(n-m)^{\gamma-1}\le \frac{1}{\Gamma(\gamma)} \int_{t_{m}}^{t_{m+1}}
(nk-s)^{\gamma-1} ds.
\]
\end{proof}

This lemma recovers the discrete Gr\"{o}nwall inequality in \cite{dixon86}:
\begin{cor}[\cite{dixon86}]
Let $\{a_n\}$ be a non-negative sequence.  If $\{a_n\}$ satisfies $$
a_n\le B+\frac{\lambda}{\Gamma(\gamma)}k^{\gamma}\sum_{m=0}^{n-1}(n-m)^{\gamma-1}a_m,~0\le n\le N,
$$
where $B>0$ and $\lambda>0$ are independent of $n, k, \gamma$, then, $$
a_n\le u(nk)=BE_{\gamma}(\lambda(nk)^{\gamma}),~0\le n\le N.
$$
Here, $u(t)$ is the solution to $D_c^{\gamma}u=\lambda u$ with initial value $B$ and $E_{\gamma}$ is the Mittag-Leffler function.
\end{cor}

We conclude the following stability result about the schemes, which is useful when studying numerical schemes of fractional PDEs.
\begin{pro}\label{pro:intnum}
Let $f(u)$ be nonnegative, non-decreasing, locally Lipschitz on $[0,\infty)$ and $u_0>0$. Suppose $u^n$ solves the numerical scheme \eqref{eq:schemeInt} or \eqref{eq:schemeInt2}, and $u$ is the unique solution to \eqref{eq:ode1} with initial value $u_0$. Then, we have \[
u^{n-1}\le u^n\le u(nk), ~1\le n< T_b/k.
\]
\end{pro}
\begin{proof}
$u^n\le u(nk)$ follows directly from Lemma \ref{lmm:discretegronwall1}.

Now, let us show that $\{u^{n}\}$ is non-decreasing under the scheme \eqref{eq:schemeInt} by induction. For $n=1$, it is clear that $u^1\ge u_0$ by direct computation. Now, let $n\ge 2$ and assume that $u^m\ge u^{m-1}$ for all $1\le m\le n-1$. 
\begin{multline*}
u^{n}=u_0+\frac{k^{\gamma}}{\Gamma(1+\gamma)}\sum_{m=0}^{n-1} f(u^{m})\Big((n-m)^{\gamma}-(n-m-1)^{\gamma}\Big)\\
\ge u_0+\frac{k^{\gamma}}{\Gamma(1+\gamma)}\sum_{m=1}^{n-1} f(u^{m-1})\Big((n-m)^{\gamma}-(n-m-1)^{\gamma}\Big)\\
= u_0+\frac{k^{\gamma}}{\Gamma(1+\gamma)}\sum_{m=0}^{n-2} f(u^{m})\Big((n-1-m)^{\gamma}-(n-m-2)^{\gamma}\Big)
=u^{n-1}.
\end{multline*} 
The proof for scheme \eqref{eq:schemeInt2} is similar.
\end{proof}

\subsection{Schemes for the differential equation}

Now, let us discretize Equation \eqref{eq:ode1} directly.  We assume the solution is $C^1(0,T_b)$, and use the following first order scheme from \cite{lx07, llx11} based on the explicit formula \eqref{eq:captra}:
\begin{multline*}
D_c^{\gamma}u(t_{n+1})=\frac{1}{\Gamma(1-\gamma)}\sum_{m=0}^n\int_{t_m}^{t_{m+1}}\frac{(u(t_{m+1})-u(t_m))/k}{(t^{n+1}-s)^{\gamma}}ds+O(k^{2-\gamma})\\
=\frac{1}{\Gamma(2-\gamma)k^{\gamma}}\sum_{m=0}^{n}
(u(t_{n+1-m})-u(t_{n-m}))[(m+1)^{1-\gamma}-m^{1-\gamma}]+O(k^{2-\gamma}).
\end{multline*}
Denote
\begin{multline}
(\mathcal{D}_h^{\gamma}u)_{n+1}=\frac{1}{\Gamma(1-\gamma)}\sum_{m=0}^n\int_{t_m}^{t_{m+1}}\frac{(u(t_{m+1})-u(t_m))/k}{(t_{n+1}-s)^{\gamma}}ds\\
=\frac{1}{\Gamma(2-\gamma)k^{\gamma}}\sum_{m=0}^{n}
(u(t_{n+1-m})-u(t_{n-m}))[(m+1)^{1-\gamma}-m^{1-\gamma}]=k^{-\gamma}\sum_{m=0}^{n+1}b^{n+1}_m u(t_{n+1-m}).
\end{multline}
We can determine that for all $n\ge 0$,
\begin{gather}
\begin{split}
&\Gamma(2-\gamma)b_0^{n+1}=1, \\
&\Gamma(2-\gamma)b_m^{n+1}=(m+1)^{1-\gamma}-2m^{1-\gamma}+(m-1)^{1-\gamma},~ 1\le m\le n, \\
&\Gamma(2-\gamma)b_{n+1}^{n+1}=-(n+1)^{1-\gamma}+n^{1-\gamma}.
\end{split}
\end{gather}
It is clear that $b_{m}^{n+1}$ does not depend on $n$ if $m\le n$. Hence, for simplicity, we write \[
\Gamma(2-\gamma)b_m=
\begin{cases}
1,&m=0,\\
(m+1)^{1-\gamma}-2m^{1-\gamma}
+(m-1)^{1-\gamma}<0,~&m\ge 1.
\end{cases}
\]

Using this basic discretization, we can formulate the explicit and implicit schemes respectively as
\begin{gather}\label{eq:diffex}
(\mathcal{D}_h^{\gamma}u)_{n+1}=f(u^n)
\end{gather}
and
\begin{gather}\label{eq:diffim}
 (\mathcal{D}_h^{\gamma}u)_{n+1}=f(u^{n+1}).
\end{gather}

First of all, we discuss the explicit scheme:
\begin{gather}
(\mathcal{D}_h^{\gamma}u)_{n+1}=f(u^n)
\Leftrightarrow b_0u^{n+1}=k^{\gamma}f(u^{n})-\sum_{m=1}^n b_m u^{n+1-m}-b_{n+1}^{n+1}u^0.
\end{gather}

The following result follows from the facts $b_m<0$ for $m\ge 1$ and  $b_n+b_{n+1}^{n+1}=b_n^n$:
\begin{lmm}
Suppose $f(u)$ is nonnegative, non-decreasing on $[0,\infty)$ and $u_0 \ge 0$, then $u^n$ given by the explicit scheme \eqref{eq:diffex} is nondecreasing.
\end{lmm}

We have the following discrete comparison principles:
\begin{lmm}\label{lmm:discretecompex}
Suppose $f(u)$ is nonnegative, non-decreasing on $[0,\infty)$ and $u_0 \ge 0$. If $\{w^n\}$ is a sequence satisfying $w^0\ge u_0$ and \[
(\mathcal{D}_h^{\gamma}w)_{m+1}\ge f(w^m),~ m\le N-1,
\]
then $w^n\ge u^n$ for $n\le N$. Correspondingly, if $w^0\le u^0$ and \[
(\mathcal{D}_h^{\gamma}w)_{m+1}\le f(w^m),~m\le N-1,
\]
then $w^n\le u^n$ for $n\le N$.
\end{lmm}
\begin{proof}
We only prove `$\ge$' case while the other case is similar. It follows directly from the following induction inequality
\begin{multline*}
b_0w^{n+1}=k^{\gamma}(\mathcal{D}_h^{\gamma}w)_{n+1}-\sum_{m=1}^n b_m w^{n+1-m}
-b_{n+1}^{n+1}w^0 \\
\ge k^{\gamma}f(w^n)-\sum_{m=1}^n b_m w^{n+1-m}
-b_{n+1}^{n+1}w^0
\ge k^{\gamma}f(u^n)-\sum_{m=1}^n b_m u^{n+1-m}
-b_{n+1}^{n+1}u^0=b_0u^{n+1},
\end{multline*}
since $b_0>0$ and $b_m<0$ for $m\ge 1$.
\end{proof}

We now move on to the implicit scheme \eqref{eq:diffim}, which is given by
\begin{gather}\label{eq:implicit}
(\mathcal{D}_h^{\gamma}u)_{n+1}=f(u^{n+1})
\Leftrightarrow
b_0u^{n+1}-k^{\gamma}f(u^{n+1})=-\sum_{m=1}^n b_m u^{n+1-m}
-b_{n+1}^{n+1}u^0.
\end{gather}
Assume that $f\in C^1[0,\infty)$, non-decreasing and $b_0-k^{\gamma}f'(u_0)>0$. Hence, $f'(z)\ge 0$.
In this case, the implicit scheme is solved by finding the root of  $b_0z-k^{\gamma}f(z)=-\sum_{m=1}^n b_m u^{n+1-m}
-b_{n+1}^{n+1}u^0$ on $[u_0, M]$ where
\[
M:=\sup\{M_0>0: b_0-k^{\gamma}f'(z)\ge 0, \forall z\in [0, M_0] \}.
\]
It is clearly that $\lim_{k\to 0}M=\infty$. Hence, it is sufficient for us to find the numerical solution on $[0, M]$. If there is no root of the scheme on $[0, M]$ for $n=N^*$, then the numerical solution breaks up, and the corresponding time 
\begin{equation}
T_b(k)=N^*k
\end{equation}
 is regarded as numerical blow-up time.  

With this convention, similarly we can show that
\begin{lmm}\label{lmm:discretecompim}
Assume that $f \in C^1[0,\infty)$ is nonnegative, non-decreasing and $u_0 \ge 0$. Then, $\{u^n\}$ given by the implicit scheme \eqref{eq:diffim} is non-decreasing. Moreover, if $w^0\ge u^0$ and \[
(\mathcal{D}_h^\gamma w)_{m}\ge f(w^{m}), ~ m\le N,
\]
then $w^n\ge u^n$ for $n\le N$, $n<N^*$.

Correspondingly, if  $w^0\le u^0$ and \[
(\mathcal{D}_h^\gamma w)_{m}\le f(w^{m}),~ m\le N,
\]
then $w^n\le u^n$ for $n\le N$, $n<N^*$.
\end{lmm} 

Combining these facts, we have the following claim
\begin{thm}\label{thm:diffnum}
Assume that $f$ is nonnegative, non-decreasing and $u_0 \ge 0$ on $[0,\infty)$. In addition, suppose $f\in C^1[0,\infty)$, $f'$ is locally Lipschitz and the solution $u$ to \eqref{eq:ode1} is convex. Let $u_{ex}^n$ be the solution given by the explicit scheme \eqref{eq:diffex} and $u_{im}^n$ be given by the implicit scheme \eqref{eq:diffim}. Then $\{u_{ex}^n\}$ and $\{u_{im}^n\}$ are monotone sequences, and we have \[
u_{ex}^n\le u(nk)\le u_{im}^n, ~ n<\min(N^*, T_b/k).
\]
\end{thm}
\begin{proof}
The monotonicity for $\{u_{ex}^n\}$ and $\{u_{im}^n\}$ follows Lemma \ref{lmm:discretecompex} and Lemma \ref{lmm:discretecompim}.

We only need to show $u_{ex}^n\le u(nk)\le u_{im}^n$. 
If $f\in C^1[0,\infty)$ and $f'$ is locally Lipschitz, then $u\in C^1(0, T_b)\cap C^0[0, T_b)$ by Lemma \ref{lmm:regu}.  Since $u$ is convex, $u'$ is non-decreasing. 
Denoting 
\[
w^n=u(nk),
\]
and thus
\[
(\mathcal{D}_hw)_{n+1}=\frac{1}{\Gamma(1-\gamma)}\sum_{m=0}^n\int_{t_m}^{t_{m+1}}\frac{(u(t_{m+1})-u(t_m))/k}{(t_{n+1}-s)^{\gamma}}ds.
\]

For the explicit scheme, we have 
\begin{multline*}
\int_{t_m}^{t_{m+1}}\frac{(u(t_{m+1})-u(t_m))/k}{(t_{n+1}-s)^{\gamma}}ds
\ge \int_{t_m}^{t_{m+1}}\frac{u'(t_m)}{(t_{n+1}-s)^{\gamma}}ds
= \int_{t_{m-1}}^{t_{m}}\frac{u'(t_m)}{(t_{n+1}-k-s)^{\gamma}}ds\\
\ge \int_{t_{m-1}}^{t_{m}}\frac{u'(s)}{(t_{n+1}-k-s)^{\gamma}}ds.
\end{multline*}
As a result, we have $(\mathcal{D}_hw)_{n+1}\ge D_c^{\gamma}u(t_{n})=f(u(t_n))=f(w^n)$ and the result follows from Lemma \ref{lmm:discretecompex}.

For implicit scheme, 
\begin{multline*}
f(u(t_{n+1}))-(\mathcal{D}_hw)_{n+1}=
D_c^{\gamma}u(t_{n+1})-(\mathcal{D}_hw)_{n+1}\\
=\sum_{m=0}^n\int_{t_m}^{t_{m+1}}\frac{1}{(t_{n+1}-s)^{\gamma}}\left(u'(s)-\frac{1}{k}\int_{t_m}^{t_{m+1}}u'(\tau)d\tau\right)ds\ge 0,
\end{multline*}
since $u'$ is non-decreasing. The last inequality is obtained by applying 
\begin{equation}\label{last}
\int_a^b f g dx \ge \frac{1}{b-a}\int_a^b g dx \int_a^b f dx,
\end{equation}
 if both $f$ and $g$ are non-decreasing non-negative continuous functions. In fact, there is $\xi\in (a, b)$ such that $g(\xi)=\frac{1}{b-a}\int_a^b g(x)dx$, and
\begin{multline*}
\int_a^b f(x)(g(x)-g(\xi))dx=\int _a^{\xi} f(x)(g(x)-g(\xi))dx+\int _{\xi}^{b} f(x)(g(x)-g(\xi))dx\\
\geq f(\xi)\int _a^{b} (g(x)-g(\xi))dx=0.
\end{multline*}
The claim then follows from Lemma \ref{lmm:discretecompim}.
\end{proof}

\begin{rmk}
The explicit schemes can be used to prove the stability and convergence of some approximation schemes for fractional PDEs and thus the convergence and existence of solutions. The implicit schemes can be used to prove positivity of solutions and to estimate the blow-up time. 
\end{rmk}

\subsection{Numerical simulations}

For $f(u)=Au^p$, the numerical solutions using explicit schemes \eqref{eq:schemeInt}, \eqref{eq:schemeInt2} and \eqref{eq:diffex} never break up (i.e. $u^n$ can be computed for any $n\ge 1$). The implicit scheme is more suitable for the study of blowup. If we use the implicit scheme \eqref{eq:diffim}, we look for the root of the scheme \eqref{eq:implicit} in $[u^n, M]$ to find $u^{n+1}$ where $M=(k^{-\gamma}/(pA\Gamma(2-\gamma))^{1/(p-1)}$. Suppose that the sequence terminates at $N^*$ and numerically we set $T_b(k)=N^* k$. It is expected that $T_b(k)\to T_b$ as $k\to 0^+$. 

For $p=2$, the implicit scheme \eqref{eq:implicit} can be solved exactly and therefore this allows us to compute the numerical solutions accurately and fast enough. Below, we do the numerical simulations using the implicit scheme for $f(u)=u^2$ by choosing $k$ sufficiently small.

\begin{figure}
	\begin{center}
		\includegraphics[width=0.8\textwidth]{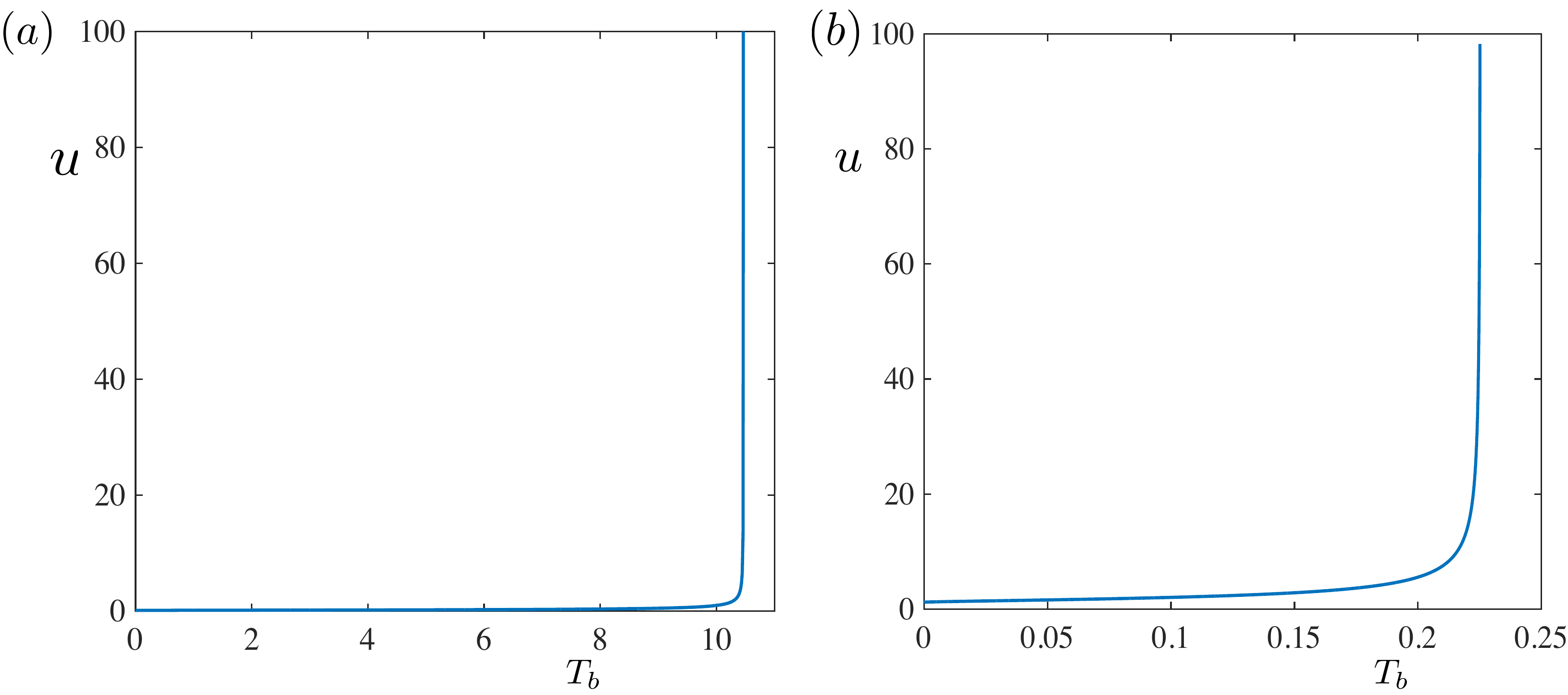}
	\end{center}
	\caption{Solution curves for $f(u)=Au^2$ with $u(0)=u_0$. (a). $A=1, u_0=0.12 ,\gamma=0.6$; (b). $A=1, u_0=1.2, \gamma=0.6$. }
	\label{fig:rep}
\end{figure}

In Figure \ref{fig:rep}, we sketch two typical solution curves. Figure \ref{fig:rep} (a) shows the solution curve with $u_0=0.12, \gamma=0.6$, while Figure \ref{fig:rep} (b) shows the solution curve with $u_0=1.2, \gamma=0.6$. Comparing the blow-up time in both cases, we find clearly that small $u_0$ defers the blowup while large $u_0$ accelerates the blowup.

\begin{figure}
	\begin{center}
		\includegraphics[width=0.8\textwidth]{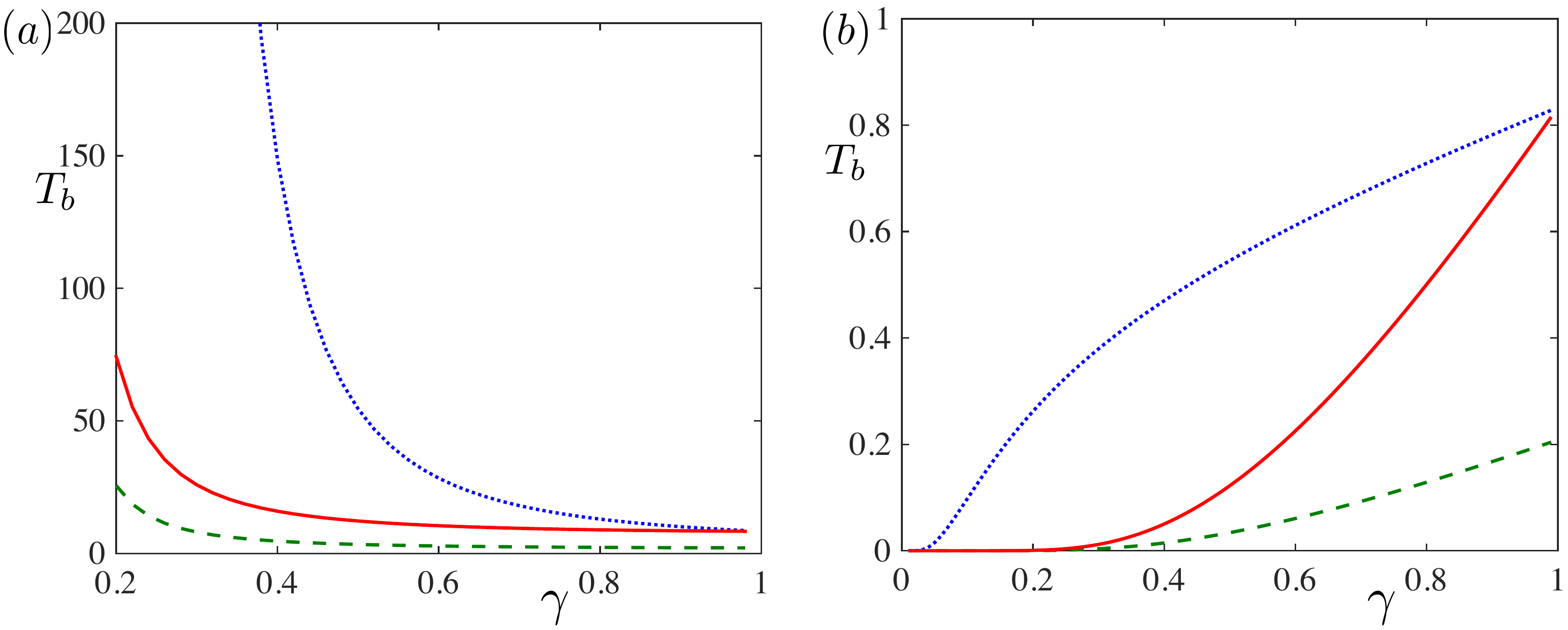}
	\end{center}
	\caption{Blow-up time versus $\gamma$. The red solid line shows the numerical results of the blow-up time. The blue dotted line is the estimated upper bound and the green dashed line is the lower bound, provided by Theorem \ref{buT}. (a).  $A=1, u_0=0.12$; (b). $A=1, u_0=1.2$. }
	\label{fig:blow-uptime}
\end{figure}
To investigate this issue further, in Figure \ref{fig:blow-uptime}, we plot the blow-up time versus $\gamma$, meanwhile we also plot the estimated upper and lower bounds gained from Theorem \ref{buT}. 
In the case $u_0=0.12$, the line of real blow-up time around $\gamma =0.2$ is quite steep. In the case  $u_0=1.2$, the line of real blow-up time around $\gamma=0$ is approximately equal to 0. 
The numerical results agree with our analysis in Section \ref{subsec:bounds}. If $u_0$ is big enough such that $u-u^2=u_0$ has no solution, i.e. $u_0>0.25$, the blow-up time decreases as $\gamma$ decreases, which means samller $\gamma$ accelerates the blowup. However, if $u_0$ is small, i.e. $u_0<0.25$, then the blow-up time increases as $\gamma$ decreases, which means samller  $\gamma$ defers the blowup. These observations agree with intuition that as the smaller $\gamma$ is, the stronger the memory effect is. 

\section{Appendix}\label{sec:app}
In this section, we restate the result in \cite{roberts1996growth} regarding the growth rate of \eqref{eq:ode2}, as we mentioned in Remark \ref{rate}. The statement is tailored to our problem, and we also present the proof for convenience.
In fact, we have the following statement.
\begin{pro}[\cite{roberts1996growth}]
For $p>1$, $\gamma\in (0,1)$, $A>0$, $u_0>0$, the solution of \eqref{eq:ode2} satisfies \eqref{ratea}, where $T_b$ is the blow-up time guaranteed by Theorem \ref{thm:increasing}.
\end{pro}
\begin{proof}
We set $v(t)=u(t)-u_0$. First, we use the following transformation:
\begin{equation}
\eta(t)=(T_b-t)^{-1}-\eta_0,\quad \eta_0=T_b^{-1},\quad \omega(\eta)=v(t).
\end{equation}
Now based on the definition of $T_b$, we have $\omega(\eta)\rightarrow \infty$ as $\eta\rightarrow \infty$. The corresponding equation for $\omega$ is as follows:
\begin{equation}\label{23}
\omega(\eta)=\frac{A}{\Gamma(\gamma)}\int_0^{\eta}(\eta-\xi)^{\gamma-1}(\xi+\eta_0)^{1-\gamma}(\eta+\eta_0)^{1-\gamma}(\xi+\eta_0)^{-2}(\omega(\xi)+u_0)^pd\xi,
\end{equation}
where
\begin{equation}
\Phi(\xi)=(\xi+\eta_0)^{-2}(\omega(\xi)+u_0)^p.
\end{equation}
Now let $\xi=\eta\tau$, then 
\begin{equation}
\omega(\eta)=\frac{A}{\Gamma(\gamma)}\eta\int_0^1\eta^{\gamma-1}(1-\tau)^{\gamma-1}(\eta\tau+\eta_0)^{-1-\gamma}(\eta+\eta_0)^{1-\gamma}(\omega(\eta\tau)+u_0)^pd\tau.
\end{equation}
Now based on \cite{bleistein1975asymptotic}, the right hand side as $\eta\rightarrow \infty$ has the following asymptotic behavior:
\begin{align*}
\omega(\eta)\sim \eta\frac{A}{\Gamma(\gamma)}\left(\frac{\eta}{\eta+\eta_0}\right)^{\gamma-1}\int_0^{\infty}K(\tau)F(\eta\tau)d\tau\sim \eta \frac{A}{\Gamma(\gamma)}\int_0^{\infty}K(\tau)F(\eta\tau)d\tau,
\end{align*}
where
\begin{align*}
K(\tau)=(1-\tau)^{\gamma-1}\theta(1-\tau),\quad F(\eta\tau)=(\eta\tau+\eta_0)^{-1-\gamma}(\omega(\eta\tau)+u_0)^p.
\end{align*}
Here $\theta(s)$ is the standard Heaviside function. As in \cite{bleistein1975asymptotic}, we use Parseval formula and Mellin transform, then
\begin{align*}
\omega(\eta)\sim \eta  \frac{A}{\Gamma(\gamma)} \frac{1}{2\pi i}\int_{c-i\infty}^{c+i\infty} M[K(\tau);1-z]M[F(\eta\tau);z]dz,
\end{align*}
where 
$$
M(\nu(\tau);z)=\int_0^{\infty} \tau^{z-1}\nu(\tau)d\tau.
$$
Now notice that
$$
M[F(\eta\tau);z]=\eta^{-z}M(F(\tau);z),
$$
and
$$
M[K(\tau);1-z]=\frac{\Gamma(\gamma)\Gamma(1-z)}{\Gamma(1+\gamma-z)}.
$$
Hence,
$$
\omega(\eta)\sim \eta \frac{A}{2\pi i}\int_{c-i\infty}^{c+i\infty}\eta^{-z}\frac{\Gamma(1-z)}{\Gamma(1+\gamma-z)}M[F(\tau);z]dz.
$$
By pluging in the anzats $\omega(\eta)\sim C\eta^{l}$ and checking the simple pole of the integrand, we have
$$
\omega(\eta)\sim \frac{A\Gamma(pl-\gamma)}{\Gamma(pl)}\omega(\eta)^{p}\eta^{-\gamma} \mbox{ as }\eta\rightarrow \infty,
$$
which is what we need.
\end{proof}

\section*{Acknowledgements}
The work of J.-G Liu is partially supported by KI-Net NSF RNMS11-07444 and NSF DMS-1514826. Y. Feng is supported by NSF DMS-1252912.

\bibliographystyle{unsrt}
\bibliography{nonlinearFODE}
\end{document}